\documentclass[11pt]{amsart}
\usepackage{amsmath,color,amsfonts,amssymb,amsthm,epsfig,amscd,epsfig,psfrag,comment,hyperref,latexsym,amsthm,graphicx,enumerate,bbm}
\usepackage[all]{xy}

\newtheorem{Theorem}{Theorem}[section]
\newtheorem{Proposition}[Theorem]{Proposition}
\newtheorem{Lemma}[Theorem]{Lemma}

\newtheorem{Corollary}[Theorem]{Corollary}
\newtheorem{Conjecture}[Theorem]{Conjecture}

\theoremstyle{definition}
\newtheorem{Definition}[Theorem]{Definition}

\theoremstyle{definition}

\newtheorem{Remark}[Theorem]{Remark}

\newcommand{\scs}{\scriptstyle}
\newcommand{\qbins}[2]{
\left[ \;
\vcenter{\xy
(0,3)*{\scs #1}; (0,-1)*{\scs #2};
\endxy} \;
 \right]
}

\def\tf{{\tilde{f}}}
\def\tg{{\tilde{g}}}
\def\dotimes{{\stackrel{\dagger}{\otimes}}}
\def\sA{\mathcal{A}}

\def\fC{{\mathfrak{C}}}
\def\fZ{{\mathfrak{Z}}}
\def\MHM{{\operatorname{MHM}}}
\def\HM{{\operatorname{HM}}}
\def\fG{{\sf{G}}}
\def\oZ{\overline{Z}}

\def\d{\dagger}

\def\IC{\mathrm{IC}}

\def\sA{\mathcal{A}}

\def\IC{{\mathrm{IC}}}

\def\tr{{\mathrm{tr}}}

\newcommand{\Z}{\mathbb{Z}}
\newcommand{\C}{\mathbb{C}}

\def\Rees{\mbox{Rees}}
\def\N{{\mathbb{N}}}

\def\gr{{\rm{gr}}}
\def\tgr{\vec{\rm{gr}}}
\def\hi{\hat{i}}
\def\ti{\tilde{i}}
\def\tp{\tilde{p}}
\def\tpi{\tilde{\pi}}
\def\tq{\tilde{q}}

\def\id{{\mathrm{Id}}}
\def\sl{{\mathfrak{sl}}}
\def\g{{\mathfrak{g}}}
\def\l{{\lambda}}
\def\E{{\sf{E}}}
\def\F{{\sf{F}}}
\def\G{{\mathbb{G}}}

\def\T{{\sf{T}}}

\def\sE{{\mathcal{E}}}
\def\sF{{\mathcal{F}}}

\def\sT{{\mathcal{T}}}

\def\sL{{\mathcal{L}}}
\def\sP{{\mathcal{P}}}

\def\O{{\mathcal O}}

\def\P{{\mathbb{P}}}

\def\sQ{{\mathcal{Q}}}

\def\bD{\mathbb{D}}

\def\D{{\mathcal{D}}}
\def\1{\mathbbm{1}}
\def\la{\langle}
\def\ra{\rangle}

\newcommand{\Cone}{\mathrm{Cone}}
\def\dim{{\rm{dim}}}

\newcommand{\K}{\mathcal{K}}

\DeclareMathOperator{\Hom}{Hom}
\DeclareMathOperator{\Ext}{Ext}

\DeclareMathOperator{\End}{End}

\DeclareMathOperator{\Sym}{Sym}
\DeclareMathOperator{\Sp}{Sp}
\DeclareMathOperator{\DR}{DR}
\DeclareMathOperator{\Ind}{Ind}
\renewcommand{\mod}{{\operatorname{-mod}}}

\title[Associated graded of Hodge modules]{Associated graded of Hodge modules and categorical $\sl_2$ actions}

\author{Sabin Cautis}
\email{cautis@math.ubc.ca}
\address{Department of Mathematics\\ University of British Columbia \\ Vancouver BC, Canada}

\author{Christopher Dodd}
\email{csdodd2@illinois.edu}
\address{University of Illinois at Urbana-Champaign \\Urbana Il, USA}

\author{Joel Kamnitzer}
\email{jkamnitz@math.utoronto.ca}
\address{Department of Mathematics \\ University of Toronto \\ Toronto ON, Canada}

\begin{document}

\begin{abstract}
One of the most mysterious aspects of Saito's theory of Hodge modules are the Hodge and weight filtrations that accompany the pushforward of a Hodge module under an open embedding. In this paper we consider the open embedding in a product of complementary Grassmannians given by pairs of transverse subspaces. The push-forward of the structure sheaf under this open embedding is an important Hodge module from the viewpoint of geometric representation theory and homological knot invariants.


We compute the associated graded of this push-forward with respect to the induced Hodge filtration as well as the resulting weight filtration. The main tool is a categorical $\sl_2$ action on the category of $\D_h$-modules on Grassmannians. Along the way we also clarify the interaction of kernels for $\D_h$-modules with the associated graded functor. Both of these results may be of independent interest.

\end{abstract}

\setcounter{tocdepth}{1}

\maketitle
\tableofcontents

\section{Introduction}

\subsection{Mixed Hodge modules and their associated graded}
The theory of mixed Hodge modules was developed by Saito as a $\D$-module analog of the theory of weights for $\ell$-adic sheaves.  Let $ X $ be a smooth proper algebraic variety.  A mixed Hodge module $ M$ on $ X$ gives rise to filtered $ \D_X$-module $ \fG(M)$ and thus to a $ \O_{T^* X}$-module $\gr \fG(M) $.

The most non-trivial aspect of Saito's theory concerns push-forward of mixed Hodge modules along non-proper maps.  For example, suppose that $ j : U \hookrightarrow X $ is the inclusion of an open subset and consider the mixed Hodge module push-forward $ j_* \O_U$.  This results in a subtle filtered structure on $j_* \O_U$ which reflects the singularities of the complement $ X \setminus U $. More precisely, $j_* \O_U$ has a simple description if the complement is simple normal crossing. For more complicated complements on has to blow up along $X \setminus U$ to obtain a simple normal crossing complement and then compute the proper pushforward of $j_* \O_U$. In general this is quite difficult to do in practice. For a more general discussion see, for instance, \cite{Sc}.

In this paper we compute $ \gr(\fG(j_* \O_U)) $ for a natural open embedding $j$ that shows up in geometric representation theory and has connections to other areas, including higher representation theory and homological knot invariants. More specifically, we take $X = \G(k,N) \times \G(N-k, N) $ the product of two complimentary Grassmannians and take $ U $ to be the locus of pairs of subspaces which intersect trivially. Our main result (Theorem \ref{thm:equivopen} and Corollary \ref{cor:main}) identifies $\gr(\fG(j_* \O_U))$ as the pushforward of a line bundle of a certain locally closed subset of $ T^* X$. We also show that the resulting weight filtration on $ \gr(\fG(j_* \O_U))$ is a natural filtration from the coherent sheaves viewpoint, related to the different components of $ T^*\G(k,N) \times_{\sl_N} T^*\G(N-k, N)$.  This computation is quite indirect and uses the theory of categorical $ \sl_2 $ actions.

\subsection{$\sl_2$ actions involving Grassmannians}

Consider the quantum group $ U_q(\sl_2)$ with its usual generators $ E, F, K $.  Going back to work of Beilinson-Lusztig-MacPherson \cite{BLM}, there is a geometric incarnation of $ U_q(\sl_2) $ involving Grassmannians $ \G(k,N) $ over a finite field with $ q^2 $ elements.  In this setup, $ E $ acts on a subspace $ W $ to produce the formal sum of all hyperplanes of $ W $ and $ F $ acts of $ W $ to produce the formal sum of all subspaces in which $ W $ sits as a hyperplane.  This defines a representation of $ U_q(\sl_2)$ on the vector space $ \oplus_{k = 0}^N \C[\G(k,N)] $.

Using Grothendieck's faisceaux-fonctions correspondence, there is a natural categorification of this construction.  Consider the categories $ D_c(\G(k,N)) $ (now using Grassmannians over $\C$) of constructible sheaves.  The natural incidence correspondence $ C(k,N) \subset \G(k,N) \times \G(k-1,N) $ can be used to define a functor $ \E : D_c(\G(k,N)) \rightarrow D_c(\G(k-1,N))$ and a similar functor $ \F $ in the opposite direction.  It is not difficult to prove that these functors satisfy the defining relations of $ U_q(\sl_2)$.  This idea seems to have been known to experts for some time and first appeared in the literature in the work of Zheng (see Theorem 3.3.6 of \cite{Z1}).  By the Riemann-Hilbert correspondence this gives a categorical $ \sl_2$ action on $\D$-modules on Grassmannians.

\subsection{Cotangent bundles to Grassmannians and filtered $\D$-modules}
On the other hand, in \cite{CKL1}, we constructed a categorical $\sl_2$ action on the categories of coherent sheaves on cotangent bundles to Grassmannians. In other words, we defined kernels
$$ \sE \in D(\O_{T^* \G(k,N) \times T^* \G(k-1,N)}) \ \text{ and } \ \sF \in D(\O_{T^* \G(k-1,N) \times T^* \G(k,N)}) $$
using the conormal bundle of $C(k,N)$ and proved that they induce a categorical $\sl_2$ action.

After learning about our work, Roman Bezrukavnikov suggested to us (back in 2008) that we should be able to relate our construction from \cite{CKL1} with Zheng's construction using the machinery of filtered $\D$-modules and Saito's theory of mixed Hodge modules.  In this paper, we carry out this idea.

More precisely, we work with sheaves of $\D_{X,h}$-modules on smooth varieties $ X $, where $ \D_{X,h} $ is the Rees algebra of differential operators on $X$.  This is a sheaf of $\C[h]$-algebras and can be specialized at $ h = 0 $ to obtain $ \pi_* \O_{T^* X} $, the push-forward of the structure sheaf of the cotangent bundle. This gives us an ``associated graded'' functor $\gr: \D_{X,h}\mod \rightarrow \O_{T^* X}\mod$. We study some general properties of this functor in section \ref{sec:2} and in the context of kernels in section \ref{sec:grkernels}. In section \ref{sec:mhm}, we review some results from the theory of mixed Hodge modules.

Working with Hodge modules is useful because we have results, such as the base change theorem, which do not hold for $\D_h$-modules. On the other hand, non-flat pullback (and in particular the tensor product) of Hodge modules is not compatible with the functor $\fG$ from Hodge modules to $\D_h$-modules. In particular, this means that the kernel formalism for Hodge modules is not compatible with $\fG$. For this reason we need to consider both Hodge modules and $\D_h$-modules in this paper.

In section \ref{sec:actionDmod}, we define a categorical $\sl_2$ action on the category of $ \D_h$-modules on products of Grassmannians. The kernels for this action are the $\D_h$-module versions of Zheng's construction (after applying the Riemann-Hilbert correspondence). In section \ref{sec:gradedaction}, we prove that taking the associated graded of this action recovers our action from \cite{CKL1} (up to conjugating by some line bundles).

\subsection{The associated graded of the equivalence}
One of the main applications of categorical $\sl_2$ actions (and the reason we studied them in \cite{CKL1}) is that they can be used to construct equivalences. More precisely, given a categorical $\sl_2$ action on some categories $\D(\l), \l \in \Z$ Chuang-Rouquier \cite{CR} defined an equivalence of categories $\T: \D(\l) \rightarrow \D(-\l) $, which categorifies the quantum Weyl group element of $U_q(\sl_2)$. $\T$ is defined as the iterated cone of the ``Rickard complex''
$$ \F^{(\l)} \rightarrow \F^{(\l+1)} \E \rightarrow \F^{(\l + 2)} \E^{(2)} \rightarrow \cdots $$
If one applies this construction to the categorical $\sl_2$ actions discussed above we obtain equivalences
$$D(\D_{\G(k,N)}\mod) \xrightarrow{\T} D(\D_{\G(N-k,N)}\mod) \ \ \text{ and } \ \
D(\O_{T^* \G(k,N)}\mod) \xrightarrow{\T'} D(\O_{T^* \G(N-k,N)}\mod).$$
Remarkably, both these equivalences have simpler and more explicit descriptions.

On the $\D$-module side the equivalence $ \T $ is given by the kernel $j_* \O_U$ where 
$$U = \{ (V, V') : V \cap V' = 0 \}  \subset \G(k,N) \times \G(N-k,N)$$
 is the open locus discussed above and $j$ is its embedding. This result first appeared (without proof) in \cite{CR}. In the current paper, we give a proof and extend the result to the context of $\D_h$-modules (see Theorem \ref{thm:equivopen}). In particular, we show that $\fG(j_* \O_U)$ is the $\D_h$-module kernel that induces $\T$ and that the associated weight filtration on $\fG(j_* \O_U)$ is the same as the natural one coming from the Rickard complex defining $\T$. This result is related to a similar result (in the context of $\ell$-adic sheaves) due to Webster-Williamson \cite{WW}.

In \cite{C1} we studied the kernel for $\T'$ in the category of coherent sheaves. We proved that it concides with the (underived) push-forward $ R^0 f_* (\sL) $, where $ f $ is the inclusion of a dense, open subset $ \fZ^o \subset T^* \G(k,N) \times_{\sl_N} T^* \G(N-k,N) $ and $ \sL $ is an explicit line bundle.

In Corollary \ref{cor:main}, we prove that there is an isomorphism $\tgr(\fG(j_* \O_U)) \cong R^0 f_*(\sL) $ (here $ \tgr$ denotes the ``directed associated graded'', which is defined in section \ref{se:AssociatedGradedKernels}). This is a purely geometric result, which (as far as we know) can only be deduced using this machinery of categorical $\sl_2$ actions.

\subsection{Generalizations}
There are two possible directions for generalizing this work. In one direction, we replace $\sl_2$ by any simply-laced simple Lie algebra $\g $ and replace $ T^* \G(k,N) $ with the corresponding Nakajima quiver varieties. In \cite{CKL2} we defined a categorical $\g$ action using coherent sheaves on these quiver varieties. On the other hand, Webster \cite{W} (building on Zheng \cite{Z2}), defined a categorical $\g$ action on a quotient category of the category of equivariant $\D$-modules on an affine space related to the quiver. The results of his paper can be generalized by working with equivariant $\D_h$ modules. Again, one can relate the two constructions using the associated graded functor. In this setting, the equivalences $\T$ and $\T'$ are replaced by braid group actions \cite{CK3} (of type $\g$).

On the other hand, we can try to replace $\G(k,N) $ by any arbitrary co-minuscule flag variety $G/P$. In this context we do not expect categorical actions of any Lie algebra. However, we do expect similar equivalences
$$ D(\D_{G/P}\mod) \xrightarrow{\T} D(\D_{G/Q}\mod) \ \ \text{ and } \ \
D(\O_{T^* G/P}\mod) \xrightarrow{\T'} D(\O_{T^* G/Q}\mod)$$
where $G/Q$ denotes the opposite flag variety. We expect that $ \T, \T' $ are both given by natural complexes and by an ``open push-forward'' description, similar to the case of $\G(k,N) $.  In section \ref{sec:cominus}, we propose some precise conjectures along these lines. These equivalences for cominuscule flag varieties are of particular interest to us since they fit into our program for constructing knot homologies \cite{CK2} (that program has only been completed in type A).

\subsection*{Acknowledgements}

We would like to thank Sergey Arkhipov, Roman Bezrukavnikov, Alexander Braverman, Ivan Losev, David Nadler, Raphael Rouquier, Christian Schnell, and Ben Webster for helpful discussions. We also thank the anonymous referee for helpful suggestions.  S.C. was supported by NSERC.  J.K. was supported by a Sloan Fellowship and by NSERC.

\section{Categorical $\sl_2$ actions}
\subsection{Notation} 
In this paper we work over the complex numbers $\C$. Let $[n]=q^{n-1}+q^{n-3}+\dots +q^{1-n}$ and $\qbins{n}{k} = \frac{[n]!}{[k]![n-k]!}$ where $[n]!=[n][n-1]\dots [1]$.

By a graded category we will mean a category equipped with an auto-equivalence $\la 1 \ra$. A graded additive $\C$-linear 2-category is a category enriched over graded additive $\C$-linear categories, that is a 2-category $\K$ such that the Hom categories $\Hom_{\K}(A,B)$ between objects $A$ and $B$ are graded additive $\C$-linear categories (with finite-dimensional Hom spaces) and the composition maps $\Hom_{\K}(A,B) \times \Hom_{\K}(B,C) \to \Hom_{\K}(A,C)$ are graded additive $\C$-linear functors.

Given a 1-morphism $A$ in an additive 2-category $\K$ and a Laurent polynomial $f=\sum f_a q^a \in \N[q,q^{-1}]$ we write $\oplus_f A$ for the direct sum over $a \in \Z$, of $f_a$ copies of $A \la a\ra$.  In particular, if $f = [n]$, then we write $\oplus_{[n]} A$ to denote the direct sum $\oplus_{k=0}^{n-1} A \la n-1-2k \ra$.

An additive category is said to be idempotent complete when every idempotent 1-morphism splits. We say that the additive 2-category $\K$ is idempotent complete when the Hom categories $\Hom_{\K}(A,B)$ are idempotent complete for any pair of objects $A, B$ in $\K$, so that all idempotent 2-morphisms split.

If $\sA$ is an abelian category, then we write $ D(\sA) $ for the bounded derived category of $ \sA$.

\subsection{Categorical actions} \label{se:defstrong}
A {\it categorical $\sl_2$ action} with target $\K$ consists of the following data.

\begin{enumerate}
\item A graded, additive, $\C$-linear, idempotent complete 2-category $\K$ (the grading shift is denoted $\la \cdot \ra$).
\item For each $\l \in \Z$ an object $ \l \in \K $. We write $ \K(\l, \l') = \Hom_{\K}(\l, \l') $ for the resulting Hom categories.
\item 1-morphisms $\sE(\l) \in \K(\l-1, \l+1)$ and $ \sF(\l): \K(\l+1, \l-1)$ for $\l \in \Z$. For convenience we will sometimes omit the $\l$ when its value is irrelevant.
\item 2-morphisms
$$X(\l): \sE(\l) \la -1 \ra \rightarrow \sE(\l) \la 1 \ra \text{ and }
T(\l): \sE(\l+1) \sE(\l-1) \la 1 \ra \rightarrow \sE(\l+1) \sE(\l-1) \la -1 \ra.$$
\end{enumerate}

On this data we impose the following conditions.

\begin{enumerate}
\item The objects $\l \in \K$ are zero for $\l \gg 0$ and $\l \ll 0$ (``integrability'').

\item We have isomorphisms $\sE(\l)_R \cong \sF(\l) \la \l \ra$ and $\sE(\l)_L \cong \sF(\l) \la -\l \ra$.

\item If $ \l \le 0 $ then
$$\sF(\l+1) \sE(\l+1) \cong \sE(\l-1) \F(\l-1) \bigoplus_{[-\l]} \id_{\l} \in \K(\l,\l)$$
while if $\l \ge 0$ then
$$\sE(\l-1) \sF(\l-1) \cong \sF(\l+1) \sE(\l+1) \bigoplus_{[\l]} \id_\l \in \K(\l,\l).$$

\item The $X$s and $T$s satisfy the nil affine Hecke relations:
\begin{enumerate}
\item $T(\l)^2 = 0$
\item $(I  T(\l-1)) \circ (T(\l+1)  I) \circ (I  T(\l-1)) = (T(\l+1) I) \circ (I T(\l-1)) \circ (T(\l+1)  I)$ as endomorphisms of $ \sE(\l+2)\sE(\l)\sE(\l-2)$.
\item $(X(\l+1)  I) \circ T(\l) - T(\l) \circ (I  X(\l-1)) = I = - (I  X(\l-1)) \circ T(\l) + T(\l) \circ (X(\l+1) I)$ as endomorphisms of $ \sE(\l+1)\sE(\l-1) $.
\end{enumerate}

\item For $\l \in \K$ nonzero the endomorphism space $\Hom(\id_\l, \id_\l \la i \ra)$ of the identity 1-morphism $\id_\l \in \K(\l,\l)$ is zero if $i < 0$ and one-dimensional if $i=0$. Moreover, the space of 2-morphisms between any two 1-morphisms is finite dimensional.
\end{enumerate}

\begin{Remark}
The definition above of a categorical $\sl_2$ action is, on the surface, a weaker version of the 2-functor from Lauda's 2-category \cite{Ld} to $\K$. The difference is that it omits certain relations among 2-morphisms involving both $\sE$'s and $\sF$'s. However, by \cite{CL} and also \cite{Br} one can recover the full action of Lauda's category.
\end{Remark}

\begin{Remark}
Typically, the 2-category $ \K $ is the 2-category of ($\C$-linear, additive, graded) categories where the 1-morphisms are functors and the 2-morphisms are natural transformations.  In this paper, $\K$ will be a 2-category of kernels where the 1-morphisms are kernels and the 2-morphisms are morphisms between kernels.
\end{Remark}

The action of the affine nilHecke algebra gives us a direct sum decomposition
$$\sE^r(\l) \cong \bigoplus_{[r]!} \sE^{(r)}(\l)$$
for some new 1-morphisms $\sE^{(r)}(\l) \in \K(\l-r,\l+r)$ (and similarly we get $\sF^{(r)}(\l) \in \K(\l+r,\l-r)$. In most geometric examples (including this paper) these $\sE^{(r)}$ are quite natural and interesting in themselves.

\section{$\D_{X,h}$-modules and the associated graded functor} \label{sec:2}
\subsection{Notation} \label{se:notation}
In this section we assume our varieties are smooth, quasi-projective and defined over $\C$. For a variety $X$, we fix the following notation:
\begin{itemize}
\item $d_X := \dim X $, the dimension,
\item $\O_X$, the structure sheaf,
\item $\Theta_X$, the tangent sheaf, and its exterior powers $ \Theta^i_X $,
\item $\Omega_X$, the cotangent sheaf, and its exterior powers $ \Omega^i_X $,
\item $\omega_X := \Omega^{d_X}_X $,
\item $\D_X $, the sheaf of differential operators.
\end{itemize}
  For any abelian category $ \mathcal A $ (such as $ \O_X\mod$ or $\D_X\mod$), we write $ D(\mathcal A) $ for the bounded derived category of $ \mathcal A $.  All functors will be assumed to be derived unless otherwise noted.

Usually in this paper, we consider $ \O$-modules on cotangent bundles $ T^* X$.  We write $ \O_{T^* X}\mod^{\C^\times} $ for the category of $ \C^\times $-equivariant coherent $ \O_{T^* X}$-modules on $ T^* X $, where $ \C^\times $ acts by scaling the fibres with weight 2.  We write $ \{1 \} : D(\O_{T^*X}\mod^{\C^\times}) \rightarrow D(\O_{T^*X}\mod^{\C^\times}) $ for the functor given by tensoring by the trivial line bundle carrying a $ \C^\times $ action of weight 1.

Recall that $\D_X$ comes equipped with a filtration $F^0 \D_X \subset F^1 \D_X \subset \dots$ by the order of differential operator.  We let $\D_{X,h} := \Rees(\D_X) = \bigoplus_k h^k F^k \D_X$ be the associated Rees algebra.  Thus, $\D_{X,h}$ is a sheaf of graded algebras which contains $\O_X$ in degree zero and $ \Theta_X $ and the parameter $h$ in degree $2$.  For $ \xi \in \Theta_X $ and $ f \in \O_X $, we have the relations $ \xi f - f \xi = h \xi(f) $.  We denote by $\D_{X,h}\mod$ the category of coherent graded $\D_{X,h}$-modules.  We write $ \{1\} $ for the functor of grading shift.

If $(M,F)$ is a $\D_X$-module with a compatible filtration, then one can form the Rees module $\Rees(M) = \bigoplus_k h^k F^k M$ which is a $\D_{X,h}$-module. It is not hard to check that $\Rees(M) \otimes_{\C[h]} \C_1 \cong M$ while $\Rees(M) \otimes_{\C[h]} \C_0 \cong \gr(M)$ is the associated graded of $M$.

In particular, $ \O_X $ is a $ \D_X$-module.  We define a filtration on $ \O_X $ by
 $$
 F^k \O_X = \begin{cases}
 &0 \text{ if } k < 0 \\
  &\O_X \text{ if } k \ge 0
 \end{cases}
 $$
  and we write $ \O_{X,h} := \Rees(\O_X) $.  Of course, $ \O_{X,h}$ is a sheaf of algebras and we have $ \O_{X,h} = \O_X \otimes_\C \C[h] $.  We write $ \O_{X,h}\mod $ for the category of quasi-coherent graded $\O_{X,h}$-modules.  Given any $ \O_X$-module $ \mathcal F $, we will write $ \mathcal F_h := \mathcal F \otimes_\C \C[h] $; this is not a $ \D_{X,h} $-module, just a graded $ \O_{X,h} $-module.
 
\begin{Remark} \label{re:connect}
Given $ M \in \D_{X,h}\mod$, we can form the connection map
\[
\nabla: M \to \Omega_{X,h} \otimes_{\O_{X,h}} M \{2\}.
\]

This is a map of sheaves of graded $ \C[h]$-modules satisfying
\begin{enumerate}
	\item $ \nabla(fm)=hdf \otimes m +f \nabla(m) $ for any $ f \in \O_{X,h} $ and $ m \in M $ (the $h$-Leibniz rule)
\item $\nabla\circ\nabla=0$ (flatness condition)
\end{enumerate}

This sets up an equivalence of categories between $ \D_{X,h}\mod $ and the category of pairs $ (\mathcal F, \nabla) $ where $ \mathcal F \in \O_{X,h}\mod $ and $ \nabla $ is a graded $ \C[h]$-module morphism satisfying the above conditions.
\end{Remark}


Returning to the definition of $\D_{X,h}$ we have that
$$\D_{X,h} \otimes_{\C[h]} \C_0 \cong \pi_{*} \O_{T^* X}$$
(where $\pi: T^*X \rightarrow X$ is the natural projection and $\C_0 = \C[h]/h$). Therefore we have a functor
$$D(\D_{X,h}\mod) \rightarrow D(\pi_{*} \O_{T^*X}\mod), \ \  M \to M \otimes_{\C[h]} \C_0.$$
If $M \in \D_{X,h}\mod $ then $M$ is graded by definition and so $ M \otimes_{\C[h]} \C_0 $ will be a graded $ \pi_* \O_{T^* X} $-module. This is the same as a $ \C^\times$-equivariant $ \O_{T^* X}$-module, since localization gives an equivalence $\pi_{*} \O_{T^*X} \mod \xrightarrow{\sim} \O_{T^*X}\mod$.  Thus, we get a functor
$$\gr: D(\D_{X,h}\mod) \rightarrow D(\O_{T^*X}\mod^{\C^\times}).$$

One also has a similar, but simpler, functor $\D_{X,h}\mod \rightarrow \D_X\mod$ given by $M \rightarrow M \otimes_{\C[h]} \C_1$.  (In this case, setting $ h $ to be 1 turns the grading on $ M $ into a filtration.)

\begin{Remark}
There are three grading shifts that show up in this paper. The first is the internal shift $\{1\}$ mentioned above. The second is the cohomological grading shift $[1]$. The third is $\la 1 \ra := [1]\{-1\}$ which appears naturally when working with the categorical $\sl_2$ action (as well as in other contexts such as $\IC$ Hodge modules).
\end{Remark}

The following derived version of the Nakayama Lemma will be useful.

\begin{Lemma}[Nakayama Lemma]
\label{lem:(Graded-Nakayama-Lemma)}
Let $M$ be a bounded complex of quasi-coherent graded $\O_{X,h}$-modules.
Suppose that, for each $i \in \Z$, the set of non-zero graded components $ \{ j \in \Z :  \mathcal{H}^{i}(M)_{j} \ne 0 \} $ is bounded below.

 Then $M \otimes_{\C[h]} \C = 0$ implies $M=0$. 
\end{Lemma}

\begin{proof}
Since this statement is local we suppose that $X$ is affine and that $M$ is a bounded complex of graded $\O(X)[h]$-modules. If $M$ is nonzero, we may consider $i$, the maximal index for which $\mathcal{H}^{i}(M)$ is nonzero. For this $i$, we have 
\[
\mathcal{H}^{i}(M\otimes_{\C[h]}\C) \cong \mathcal{H}^{i}(M)\otimes _{\C[h]}\C
\]
where the left hand side denotes the derived tensor product and the right hand side denotes the underived tensor product (this is a general fact about the derived tensor product). But the assumption on $\mathcal{H}^{i}(M)$, and the usual graded Nakayama lemma, imply that 
\[
\mathcal{H}^{i}(M)\otimes _{\C[h]}\C \neq 0
\]
contradicting the assumption that $M \otimes_{\C[h]} \C = 0$. Thus in fact $\mathcal{H}^{i}(M)=0$ for all $i$ as required.
\end{proof}

In the rest of this section we will explain how the functor $\gr$ commutes with pullbacks and pushforwards of $\D_h$-modules. These results first appeared in a paper by Laumon \cite{Lm} using the language of exact categories and filtered modules. We reinterpret and reprove these results using the language of $\D_h$-modules. We also discuss Verdier duality and adjunctions in the context of $\D_h$-modules.

\subsection{Pull-back}\label{sec:pull}
We begin with a discussion of  the pullback functor in the setting of $\D$-modules. We consider a map $f: X \rightarrow Y$ between smooth varieties. For a $\D_Y$-module $N$ we have the usual $\O$-module pullback $f^* N$, which is then a $f^* \D_Y$-module.  Since $f$ induces a natural morphism $\Theta_X \to f^{*}\Theta_Y$ the sheaf $f^* N$ acquires an action of $\D_X$. Following \cite[Sect. 1.3]{HTT} we rewrite this functor as
$$f^{*} N \cong f^{*} \D_Y \otimes_{f^{-1}\D_{Y}} f^{-1} N.$$
where the $\D_X$ action on the tensor product is induced from that on $f^{*} \D_Y$.  Here and below, we will use $ \otimes $ to denote the derived functor of tensor product.

We now extend this to $ \D_h$-modules in the obvious way.
\begin{Definition} \label{def:pull}
For $ N \in D(\D_{Y,h})\mod $,  we define the \textbf{pull-back} of $ \D_h$-modules by
$$f^{*} N := f^* \D_{Y,h} \otimes_{f^{-1}\D_{Y,h}} f^{-1} N.$$
where $f^{*}\D_{Y,h}$ is a $\D_{X,h}$ module again by the map $\Theta_X \rightarrow f^* \Theta_Y$. This functor is left exact and cohomologically bounded which allows us to define $f^\dagger: D(\D_{Y,h}\mod) \rightarrow D(\D_{X,h}\mod)$ as
$$f^{\dagger} N := f^{*} N [d_X-d_Y].$$
\end{Definition}
This shift is useful since $f^\d$ appears in the base change formula for $ \D_h$-modules.

To describe what happens when we specialize to $h=0$, consider the following commutative diagram associated to $f: X \rightarrow Y$.
\begin{equation}\label{eq:1}
\xymatrix{
T^* X \ar[d]^{\pi_X} & & \ar[dll]_{p_1} \ar[ll]_{\pi_f} \ar[d]^{p_2} X \times_Y T^* Y \ar[rr]^{p_1} & & X \ar[d]^f \\
X & & T^* Y \ar[rr]^{\pi_Y} & & Y
}
\end{equation}

\begin{Proposition}\label{prop:pull}
For $N \in D(\D_{Y,h}\mod)$ we have $\gr(f^* N) \cong \pi_{f*} p_{2}^{*} (\gr N)$.
\end{Proposition}
\begin{proof}
First, we have
\begin{align*}
\pi_{X*} (f^{*} N \otimes_{\C[h]} \C_0)
&\cong \pi_{X*} [ (f^* \D_{Y,h} \otimes_{\C[h]} \C_0) \otimes_{f^{-1} \D_{Y,h} \otimes_{\C[h]} \C_0} (f^{-1}(N) \otimes_{\C[h]} \C_0) ] \\
&\cong f^* \Sym^* T_Y \otimes_{f^{-1} \Sym^* T_Y} f^{-1} \pi_{Y*} (N \otimes_{\C[h]} \C_0) \\
&\cong \O_X \otimes_{f^{-1} \O_Y} f^{-1} \Sym^* T_Y \otimes_{f^{-1} \Sym^* T_Y} f^{-1} \pi_{Y*} (N \otimes_{\C[h]} \C_0) \\
&\cong \O_X \otimes_{f^{-1} \O_Y} f^{-1} \pi_{Y*}(N \otimes_{\C[h]} \C_0) \\
&\cong f^{*} \pi_{Y*} (N \otimes_{\C[h]} \C_0)
\end{align*}
where the second line is a $\Sym^* T_X$-module via the natural map $\Sym^* T_X \rightarrow f^* \Sym^* T_Y$ and, also in the second line and beyond, we think of $f^{-1} \pi_{Y*}(N \otimes_{\C[h]} \C_0)$ as a $f^{-1} \Sym^* T_Y$-module. Next, since $\pi_Y$ is flat, the commutativity of the square in (\ref{eq:1}) implies
$$f^{*} \pi_{Y*} (N \otimes_{\C[h]} \C_0) \cong p_{1*}p_{2}^{*} (N \otimes_{\C[h]} \C_0).$$
The result follows since $p_{1*}=\pi_{X*} \circ \pi_{f*}$.
\end{proof}

\subsection{Push-forward}\label{sec:push}
Now consider a proper morphism $f: X \rightarrow Y$. Let us recall the usual push-forward of a $\D_X$-module $M$
\begin{align}\label{eq:push}
\int_{f} M &:= f_{*} (\D_{Y \leftarrow X} \otimes_{\D_{X}} M) \in D(\D_Y\mod)
\end{align}
where $f_{*}$ is the usual derived pushforward of quasi-coherent $\O$-modules and $\D_{Y \leftarrow X}$ is the transfer bimodule, defined as $f^{*}(\D_Y \otimes_{\O_{Y}} \omega_{Y}^{-1}) \otimes_{\O_{X}} \omega_{X}$ (where $f^{*}$ is the usual pullback).

Note that $\D_{Y \leftarrow X}$ carries the structure of a right $\D_X$-module as follows. $\D_Y$ is naturally a \emph{right} $\D_{Y}$-module by multiplication and so $\D_{Y} \otimes_{\O_{Y}}\omega_{Y}^{-1}$ is a left $\D_{Y}$-module by side-changing. Then $f^{*}(\D_{Y}\otimes_{\O_{Y}}\omega_{Y}^{-1})$ is naturally a left $\D_{X}$-module by pull-back. Thus by side-changing $f^{*}(\D_{Y} \otimes_{\O_{Y}}\omega_{Y}^{-1})\otimes_{\O_{X}} \omega_{X}$ is naturally a right $\D_X$-module. Thus, the tensor product in (\ref{eq:push}) makes sense. Moreover, the left $\D_{Y}$-module structure in (\ref{eq:push}) is inherited from the left $\D_{Y}$-module structure on $\D_{Y}$.

\begin{Definition} \label{def:push}
The sheaf $f^{*}(\D_{Y,h}\otimes_{\O_{Y,h}}\omega_{Y,h}^{-1})\otimes_{\O_{X,h}}\omega_{X,h}$ is naturally a right $\D_{X,h}$-module by the same consideration as above. We denote this module by $\D_{Y \leftarrow X,h}$ and define \textbf{push-forward} of $ \D_h$-modules by
\begin{align}\label{eq:push2}
\int_{f} M &:= f_{*}(\D_{Y \leftarrow X,h}\otimes_{\D_{X,h}} M) \in D(\D_{Y,h}\mod).
\end{align}
This is naturally a left $\D_{Y,h}$-module via the left action of $\D_{Y,h}$ on itself.
\end{Definition}

\begin{Remark} \label{re:PushFunctoriality}
	The push-forward $ \int_f $ is a functor, so it is possible to push-forward morphisms of $ \D_h$-modules.  In fact, this functoriality holds on a sheaf-theoretic level.  Let $ M, N \in \D_{X,h}\mod $.  Then there is a map
	$$f_* \mathcal Hom_{\D_{X,h}}(M,N) \rightarrow \mathcal Hom_{\D_{Y,h}}(\int_f M, \int_f N)
	$$
	To construct this map, we use the composition
	\begin{align*}
	f_* \mathcal Hom_{\D_{X,h}}(M,N) &\rightarrow 
	f_* \mathcal Hom_{f^{-1} \D_{Y,h}}(\D_{Y \leftarrow X,h} \otimes_{\D_{X,h}} M,\D_{Y \leftarrow X,h} \otimes_{\D_{X,h}} N)\\
	&\rightarrow \mathcal Hom_{\D_{Y,h}}(f_*(\D_{Y \leftarrow X,h} \otimes_{\D_{X,h}} M),f_*(\D_{Y \leftarrow X,h} \otimes_{\D_{X,h}} N))
	\end{align*}
	
	\end{Remark}

Now we look at what happens to the push-forward when we specialize to $h=0$. As before we use the diagram
\begin{equation}
\xymatrix{
 X \times_Y T^* Y \ar[d] \ar[rr]^{p_1} & & X \ar[d]^f \\
 T^* Y \ar[rr]^{\pi_Y} & & Y
}
\end{equation}

\begin{Proposition}\label{prop:push}
As $\O_X $-modules we have
$$\D_{Y \leftarrow X,h} \otimes_{\C[h]} \C_0 \cong p_{1*}(\O_{X \times_{Y} T^* Y}) \otimes_{\O_X} \omega_{X/Y}$$
where $\omega_{X/Y}$ denotes $\omega_X \otimes_{\O_X} f^* \omega_{Y}^{-1}$.
\end{Proposition}
\begin{proof}
We recall that $\D_{Y \leftarrow X,h}=f^{*}(\D_{Y,h} \otimes_{\O_{Y,h}} \omega_{Y,h}^{-1}) \otimes_{\O_{X,h}} \omega_{X,h}$. As above, we have $\D_{Y,h} \otimes_{\C[h]} \C_0 \cong \pi_{Y*} \O_{T^* Y}.$  So it follows that
$$\D_{Y\leftarrow X,h} \otimes_{\C[h]} \C_0 \cong f^*(\pi_{Y*}(\O_{T^* Y}) )\otimes \omega_{X/Y}$$
Since $ \pi_Y $ is flat, we have the base change isomorphism $f^{*}(\pi_{Y*}(\O_{T^* Y})) \cong p_{*}(\O_{X \times_Y T^* Y})$.
\end{proof}

\begin{Corollary}\label{cor:push}
For $M \in D(\D_{X,h}\mod)$ and using the notation from (\ref{eq:1}) we have
$$\gr \left( \int_{f} M \right) \cong (p_2)_* (\pi_f^*  \gr(M) \otimes p_1^* \omega_{X/Y}).$$
\end{Corollary}
\begin{proof}
Using Proposition \ref{prop:push}, we have
\begin{align*}
\left( \int_{f} M \right) \otimes_{\C[h]} \C_0
&\cong f_{*} (\D_{Y \leftarrow X,h} \otimes_{\D_{X,h}} M) \otimes_{\C[h]} \C_0 \\
&\cong f_{*} [ p_{1*}(\O_{X \times_Y T^* Y}) \otimes_{\O_X} \omega_{X/Y} \otimes_{\O_{X}} (M \otimes_{\C[h]} \C_0) ] \\
&\cong f_* p_{1*} [ p_1^* \omega_{X/Y} \otimes_{\O_{X \times_Y T^* Y}} p_1^*(M \otimes_{\C[h]} \C_0) ]
\end{align*}
where in the second and third lines we view $M \otimes_{\C[h]} \C_0$ as a $\Sym^* T_X$-module on $X$. Thus we get
$$\left( \int_{f} M \right) \otimes_{\C[h]} \C_0 \cong \pi_{Y*} (p_2)_*(\pi_f^* (M \otimes_{\C[h]} \C_0) \otimes_{\O_{X \times_Y T^*Y}} p_1^* \omega_{X/Y}).$$
The result follows.
\end{proof}

\begin{Definition} \label{def:DeltaFun}
If $f: X \rightarrow Y $ is a closed embedding, then we will use the notation $\delta_{X,h} := \int_f \O_{X,h}\{d_X\} $. This is sometimes called the sheaf of \textbf{delta-functions} along $X$.
\end{Definition}

\begin{Corollary} \label{cor:grofDelta}
Assume that $ f : X \rightarrow Y $ is a closed embedding.  Then
$$\gr(\delta_{X,h}) \cong \O_{T^*_X(Y)} \otimes p^* \omega_{X/Y}\{d_X\}$$
where $ T^*_X(Y) \subset (T^* Y)|_X $ denotes the conormal bundle of $X \subset Y$ and $p: T^*_X(Y) \rightarrow X$ is the natural projection.
\end{Corollary}
\begin{proof}
This follows immediately from Corollary \ref{cor:push} by observing that $ \pi_f^{-1}(X) = T^*_X(Y) $.
\end{proof}

\subsection{Tensor products}\label{sec:tensor}

Fix a variety $X$ as before and let $M,N \in D(\D_{X,h}\mod)$. Then one can consider the (derived) tensor product $M \otimes_{\O_{X,h}} N \in D(\D_{X,h})$ where $\D_{X,h}$ acts on the tensor using the Leibniz rule. Alternatively, we have $M \otimes_{\O_{X,h}} N = \Delta^{*}(M \boxtimes N)$ where $\Delta: X \rightarrow X \times X$ is the diagonal inclusion and $M \boxtimes N$ is the exterior tensor product. We now explain how this tensor product commutes with the associated graded functor.

Denote by $\rho: T^* X \times_X T^* X \rightarrow T^* X$ the map which takes $(p, v_1, v_2) \in T^* X \times_X T^* X$ to $(p, v_1+v_2)$ where $v_1$ and $v_2$ are covectors over $p \in X$. We also have the two natural projections $\pi_1, \pi_2: T^* X \times_X T^* X \rightarrow T^* X$.

\begin{Proposition}\label{prop:assgrtensor} Let $M,N \in D(\D_{X,h}\mod)$. Then
$$\gr (M \otimes_{\O_{X,h}} N) \cong \rho_*(\pi_1^* ( \gr M ) \otimes \pi_2^* (\gr N))  .$$
\end{Proposition}
\begin{proof}
By Proposition \ref{prop:pull} we have
$$(\Delta^*(M \boxtimes N)) \otimes_{\C[h]} \C_0 \cong \rho_* p_2^* ((M \boxtimes N) \otimes_{\C[h]} \C_0)$$
where we use the natural maps
$$T^* X \xleftarrow{\rho} X_{X \times X} T^*(X \times X) \xrightarrow{p_2} T^* X \times T^* X.$$
To see why the map $\pi_\Delta$ in (\ref{eq:1}) corresponds to $\rho$ note that $\pi_\Delta$ is induced by $T X \rightarrow T (X \times X)|_{\Delta}$ via the diagonal map and it is easy to see that the dual of this map is $\rho$.

Thus we get
\begin{align*}
(M \otimes_{\O_{X,h}} N) \otimes_{\C[h]} \C_0
&\cong (\Delta^* (M \boxtimes N)) \otimes_{\C[h]} \C_0 \\
&\cong \rho_* p_2^* ((M \boxtimes N) \otimes_{\C[h]} \C_0) \\
&\cong \rho_* p_2^* ((M \otimes_{\C[h]} \C_0) \boxtimes (N \otimes_{\C[h]} \C_0)) \\
&\cong \rho_* (\pi_1^* (M \otimes_{\C[h]} \C_0) \otimes \pi_2^* (N \otimes_{\C[h]} \C_0)).
\end{align*}
\end{proof}

We will also find it useful to define the tensor product with respect to $\d$. Namely we define $M \dotimes_{\O_{X,h}} N$ as $\Delta^\d (M \boxtimes N)$. Note that $M \dotimes_{\O_{X,h}} N = M \otimes_{\O_{X,h}} N [- d_X]$. 

\subsection{Projection formula} \label{sec:proj-formula}

Next we give a proof of the projection formula for $\mathcal{D}_h$-modules in the style of \cite[Theorem 2.3.19]{Bj}. 
In addition to the tensor product $M\otimes_{\mathcal{O}_{X,h}}N$
of two left $\mathcal{D}_{X,h}$-modules, if $M'$ is a right $\mathcal{D}_{X,h}$-module and $N$ is a left $ \D_{X,h} $-module,
the tensor product $M' \otimes_{\mathcal{O}_{X,h}}N$ is a right $\mathcal{D}_{X,h}$-module with the formula for the action given by 
\[
(m'\otimes n)\cdot\partial=-m'\partial\otimes n+m'\otimes\partial n
\]
for any derivation $\partial$ on $X$.

\begin{Proposition}\label{prop:projection} 
Let $f:X\to Y$ be proper. Then there is a canonical isomorphism
\[
\int_{f}(f^{*} N \otimes_{\mathcal{O}_{X,h}} M) \cong N \otimes_{\mathcal{O}_{Y,h}}\int_{f} M
\]
for any $M \in D(\D_{X,h}\mod)$ and  $N \in D(\D_{Y,h}\mod)$.
\end{Proposition}
\begin{proof}	
We have 
\begin{align*}
\int_{f}(f^{*}N \otimes_{\mathcal{O}_{X,h}} M) &\cong f_{*}(\mathcal{D}_{Y\leftarrow X,h}\otimes_{\mathcal{D}_{X,h}}(f^{*} N \otimes_{\mathcal{O}_{X,h}}M))
\\
&\cong f_{*}((\mathcal{D}_{Y\leftarrow X,h}\otimes_{\mathcal{O}_{X,h}}f^{*}N)\otimes_{\mathcal{D}_{X,h}}M)
\\
&\cong f_{*}((\omega_{X,h}\otimes_{\mathcal{O}_{X,h}}f^{*}(\omega_{Y,h}^{-1}\otimes_{\mathcal{O}_{Y,h}}\mathcal{D}_{Y,h})\otimes_{\mathcal{O}_{X,h}}f^{*}N)\otimes_{\mathcal{D}_{X,h}}M)
\\
&\cong f_{*}((\omega_{X,h}\otimes_{\mathcal{O}_{X,h}}f^{*}((\omega_{Y,h}^{-1}\otimes_{\mathcal{O}_{Y,h}}\mathcal{D}_{Y,h})\otimes_{\mathcal{O}_{Y,h}}N))\otimes_{\mathcal{D}_{X,h}}M)
\\
&\cong (\mathcal{D}_{Y,h}\otimes_{\mathcal{O}_{Y,h}}N)\otimes_{\mathcal{D}_{Y,h}}\int_{f}M
\\
&\cong N\otimes_{\mathcal{O}_{Y,h}}\int_{f}M
\end{align*}
where the first isomorphism is by definition, the second and fifth are easy tensor-product interchanges (as in \cite{HTT}, Prop 1.5.19), the third is the definition of $\mathcal{D}_{Y\leftarrow X,h}$, and the fourth is commutativity of pullback with tensor product.
\end{proof}

\begin{Lemma}\label{lem:local}
For any $M \in D(\D_{X,h}\mod)$ and  $N \in D(\D_{Y,h}\mod)$ we have a natural isomorphism
\[
f_{*}((\omega_{X,h}\otimes_{\mathcal{O}_{X,h}}f^{*}M)\otimes_{\mathcal{D}_{X,h}}N) \cong (\omega_{Y,h}\otimes_{\mathcal{O}_{Y,h}}M)\otimes_{\mathcal{D}_{Y,h}}\int_{f}N
\]
\end{Lemma}
\begin{proof}
By definition 
\begin{equation}
\begin{aligned} \label{eq:Lemma313}
(\omega_{Y,h}\otimes_{\mathcal{O}_{Y,h}}N)\otimes_{\mathcal{D}_{Y,h}}\int_{f}M &\cong (\omega_{Y,h}\otimes_{\mathcal{O}_{Y,h}}N)\otimes_{\mathcal{D}_{Y,h}}f_{*}(\mathcal{D}_{Y\leftarrow X,h}\otimes_{\mathcal{D}_{X,h}}M) \\
&\cong f_{*}(f^{-1}(\omega_{Y,h}\otimes_{\mathcal{O}_{Y,h}}N)\otimes_{f^{-1} \mathcal{D}_{Y,h}}\mathcal{D}_{Y\leftarrow X,h}\otimes_{\mathcal{D}_{X,h}}M)
\end{aligned}
\end{equation}
where the last isomorphism is the general projection formula for sheaves of rings, for the ringed spaces $(X,f^{-1}\mathcal{D}_{Y,h})$ and $(Y,\mathcal{D}_{Y,h})$ (cf. \cite{Stacks}, Tag 0B54). This applies here because $\mathcal{D}_{Y,h}$ has finite homological dimension (as in \cite{HTT}, appendix D.2), so that the coherent module $N$ is isomorphic to a perfect complex of $\mathcal{D}_{Y,h}$-modules. Since 
\[\mathcal{D}_{Y\leftarrow X,h}\cong\omega_{X,h}\otimes_{f^{-1}\mathcal{O}_{Y,h}}f^{-1}(\D_{Y,h}\otimes_{\mathcal{O}_{Y,h}}\omega_{Y,h}^{-1})
\]
the last line in (\ref{eq:Lemma313}) is isomorphic to 
\[
f_{*}((\omega_{X,h}\otimes_{f^{-1}\mathcal{O}_{Y,h}}f^{-1}N)\otimes_{\mathcal{D}_{X,h}}M)\cong f_{*}((\omega_{X,h}\otimes_{\mathcal{O}_{X,h}}f^{*}N)\otimes_{\mathcal{D}_{X,h}}M)
\]
as claimed.
\end{proof}

\subsection{The relative de Rham complex}
An important technical role will be played by the relative Spencer and de Rham complexes. Throughout this section and the next, let $ f : X \to Y $ be a smooth morphism. 

Since $ f $ is smooth, we have dual short exact sequences
\[
0\to f^{*}\Omega_{Y}\to\Omega_{X} \to\Omega_{X/Y} \to 0 \qquad 0 \to \Theta_{X/Y} \to \Theta_X \to f^* \Theta_Y \to 0 
\]
of locally free sheaves.
As usual, we write $ \Omega_{X/Y}^i  $ (resp. $\Theta_{X/Y}^i$) for the exterior powers of $ \Omega_{X/Y}$ (resp. $\Theta_{X/Y}$). Note that we have $ \omega_{X/Y} \cong \Omega_{X/Y}^{d_X - d_Y} $.

In order to define the de Rham complex, we begin by considering the composition
$$
\nabla_{X/Y}:\mathcal{D}_{X,h}\to \Omega_{X,h} \otimes_{\O_{X,h}} \D_{X,h} \{2\} \to
\Omega_{X/Y,h} \otimes_{\mathcal{O}_{X,h}} \mathcal{D}_{X,h}\{2\}
$$
where the first map is the $h$-connection on $ \D_{X,h} $ given by regarding it as a left $ \D_{X,h}$-module.

\begin{Definition}
We extend $ \nabla_{X/Y} $ to a complex $\DR_{Y}(\mathcal{D}_{X,h})$ of free graded right $ \D_{X,h} $-modules, called the \textbf{relative de Rham complex}, 
$$
\D_{X,h} \to \Omega_{X/Y,h}^{1} \otimes_{\O_{X,h}} \D_{X,h}\{2\}  \to \cdots \to \Omega_{X/Y,h}^{d_X - d_Y} \otimes_{\O_{X,h}} \D_{X,h}  \{2(d_X -d_Y)\}  
$$
where we place $ \D_{X,h} $ in degree 0.

The differential is given by 
\begin{align*}
\Omega_{X/Y,h}^{i} \otimes_{\mathcal{O}_{X,h}} \mathcal{D}_{X,h} & \to \Omega_{X/Y,h}^{i+1} \otimes_{\mathcal{O}_{X,h}} \mathcal{D}_{X,h} \{2 \} \\
(\eta \otimes P) & \mapsto \mu(\eta \otimes \nabla_{X/Y}(P))+h d\eta \otimes P
\end{align*}
where $\mu: \Omega_{X/Y,h}^{i} \otimes_{\O_{X,h}} \Omega_{X/Y,h}^{1} \otimes_{\O_{X,h}} \D_{X,h} \to \Omega_{X/Y,h}^{i+1} \otimes_{\O_{X,h}} \mathcal{D}_{X,h}   $ is multiplication of differential forms. 
\end{Definition}

Now consider the cohomology in the final step of this complex. Recall that inside $\mathcal{D}_{X,h}$ we
have the left ideal generated by $\Theta_{X/Y}\subset\Theta_{X}$
which yields a quotient map 
\[
\mathcal{D}_{X,h}\to\mathcal{D}_{X,h}/\mathcal{D}_{X,h}\cdot\Theta_{X/Y} \cong f^{*}\D_{Y,h}
\]
so that we obtain a map 
\[
\omega_{X/Y,h} \otimes_{\mathcal{O}_{X,h}}
\mathcal{D}_{X,h} \to \omega_{X/Y,h} \otimes_{\mathcal{O}_{X,h}} f^{*}\D_{Y,h}=\mathcal{D}_{Y\leftarrow X,h}
\]

\begin{Lemma} \label{le:deRham}
	This map gives an isomorphism
	$$\DR_{Y}(\mathcal{D}_{X,h})[d_X - d_Y] \{-2(d_X - d_Y)\} \cong \mathcal{D}_{Y\leftarrow X,h} $$
	in the category $ D(\D_{X,h}^{\mathrm{op}}\mod) $.
\end{Lemma}

\begin{proof}
	It suffices to show that the augmented
	de Rham complex obtained by placing $\mathcal{D}_{Y\leftarrow X,h}$
	in degree $d_X- d_Y + 1$ is exact. Since this is
	a complex of graded flat $\mathbb{C}[h]$-modules, all of whose nonzero
	components are concentrated in degree $\geq0$, it suffices to show
	(by Lemma \ref{lem:(Graded-Nakayama-Lemma)}) that the reduction
	mod $h$ is exact.  But this is the usual Koszul resolution argument
	(see \cite{HTT}, Lemma 1.5.27).
\end{proof}

\begin{Corollary}
	For any smooth, proper $ f: X \to Y $, we have 
	$$\int_{f} \O_{X,h} \cong 
	f_{*} \DR_Y(\mathcal{O}_{X,h})[d_{X} - d_Y]\{-2(d_{X} - d_Y)\}
$$
where $ \DR_Y(\O_{X,h}) := \DR_Y(\D_{X,h}) \otimes_{\D_{X,h}} \O_{X,h}$ 
\end{Corollary}

\begin{Remark}
This Corollary shows that our definition of the push-forward differs,
by an equivariant shift, from the usual definition one might encounter
in Hodge theory. Consider the case that $Y$ is a point.  Then this Corollary gives
$$
\int_f \O_{X,h} \cong H^\bullet(X, \Omega_{X,h}^\bullet \{2\bullet-2d\})[d]
$$
where $d = d_X$. 
The lowest term of this complex is $ H^0(X, \mathcal{O}_{X,h}) $ in homological degree $-d$ and equivariant degree $ 2d $, while the highest term is $H^{d}(X, \omega_{X,h})$  in homological degree $d$ and equivariant degree $0$. However, one would typically regard the lowest term as being in homological degree $0$ and the highest term
in degree $2d$. This convention is convenient for our purposes, as it simplifies the formulas for the push-forward map, and Verdier
duality. On the other hand, it does introduce shifts in the adjunction formula for smooth maps (as in Proposition \ref{pr:adj2} below). 
\end{Remark}

\subsection{The relative Spencer complex}
We continue with a smooth morphism $ f : X \to Y $.

\begin{Definition}
The \textbf{relative Spencer complex}, denoted $\Sp_{Y}(\D_{X,h})$,  a complex of free graded left $ \D_{X,h}$-modules. It is given by
$$
\D_{X,h} \otimes_{\O_{X,h}} \Theta_{X/Y,h}^{d_X - d_Y}\{-2(d_X - d_Y)\} \to \cdots \to \D_{X,h} \otimes_{\O_{X,h}} \Theta_{X/Y,h} \{-2\} \to \D_{X,h}
$$
where $ \D_{X,h} $ is placed in homological degree $0$.

The differential 
\[
\mathcal{D}_{X,h}\otimes_{\mathcal{O}_{X,h}}\Theta_{X/Y,h}^{i}\to\mathcal{D}_{X,h}\otimes_{\mathcal{O}_{X,h}}\Theta_{X/Y,h}^{i-1}\{2\}
\]
is given by 
\begin{align*}
d(P\otimes\theta_{1}\cdots\theta_{i})=&\sum_{j}(-1)^{j+1}P \theta_j \otimes\theta_{1}\dots\widehat{\theta}_{j}\cdots\theta_{i} \\
&+\sum_{k<j}(-1)^{k+j}P\otimes h[\theta_{k},\theta_{j}]\cdot\theta_{1}\cdots\widehat{\theta}_{k}\cdots\widehat{\theta}_{j}\cdots\theta_{i}
\end{align*}
\end{Definition}
It is straight forward to verify that $d^2=0$. Note that the image of the map $\mathcal{D}_{X,h}\otimes_{\mathcal{O}_{X,h}}\Theta_{X/Y,h} \to\mathcal{D}_{X,h}$ is precisely the left ideal generated by $\Theta_{X/Y}$. Therefore the cokernel of this map is precisely $ f^{*}\mathcal{D}_{Y,h}$.

\begin{Lemma}
\label{lem:Spencer}Via the above map, we have an isomorphism $$ \Sp_{Y}(\mathcal{D}_{X,h}) \cong f^* \D_{Y,h} $$ 
in $ D(\D_{X,h}\mod) $.
\end{Lemma}

\begin{proof}
The proof is similar to Lemma \ref{le:deRham}.
\end{proof} 

If we consider the case that $ Y $ is a point, we obtain the following.
\begin{Corollary} \label{co:SpPoint}
	We have an isomorphism
$$\Sp(\mathcal{D}_{X,h}) \cong \O_{X,h} $$ 
in $ D(\D_{X,h}\mod) $.
\end{Corollary}

\subsection{Verdier duality}\label{sec:adjoints}

The Verdier duality functor $\bD: D(\D_{X,h}\mod) \rightarrow D(\D_{X,h}\mod)^{\mathrm{op}}$ is defined by
\begin{equation}\label{eq:5}
\bD_X M = \mathcal{H}om_{\D_{X,h}}(M, \D_{X,h}) \otimes_{\O_{X,h}} \omega_{X,h}^{-1} [d_X].
\end{equation}

The following is the main result in this section.

\begin{Theorem}\label{prop:duality}
If $f: X \rightarrow Y$ is a proper morphism then $\int_f \circ \bD_X \cong \bD_Y \circ \int_f$. 
\end{Theorem}

In the language of filtered $\mathcal{D}$-modules, this result appears in \cite{Sa3}. The proof below is inspired by the one in \cite{Sa3}, though technically somewhat different. To proceed, we will first construct a natural morphism 
\[
\tr:\int_{f}\O_{X,h} [d_X-d_Y] \to \O_{Y,h} 
\]
which we use to obtain a natural transformation 
\[
\int_{f}\circ\mathbb{D}_{X}\to\mathbb{D}_{Y}\circ\int_{f}
\]
Using Nakayama's Lemma \ref{lem:(Graded-Nakayama-Lemma)}  it then suffices to check that this map is an isomorphism after applying $\gr$. We shall show that this follows from Grothendieck duality for coherent sheaves. We shall, therefore, make use of several of the basic results of this theory, as exposited in \cite{HartRD, Co}. 

In particular, let us recall that, if $f:X\to Y$ is a proper morphism of smooth varieties, there is a functor $f^{!}: D(\O_Y\mod) \to D(\O_X\mod)$ and a natural isomorphism 
\[
\mathcal{H}om_{Y}(f_{*}M,N) \xrightarrow{\sim} f_{*}\mathcal{H}om_{X}(M,f^{!}N)
\]
for  $M \in D(\O_X\mod), N \in D(\O_Y\mod)$, making $f^{!}$ a right adjoint to $f_{*}$. Taking $M=\omega_{X}[d_{X}]$, $N=\omega_{Y}[d_{Y}]$ and using that $f^!(\omega_Y [d_Y]) \cong \omega_X [d_X]$ we get
\[
\mathcal{H}om_{Y}(f_{*}\omega_{X}[d_X],\omega_{Y}[d_Y])
\cong f_* \mathcal{H}om_{X} (\omega_X [d_X], \omega_X [d_X]) 
\cong f_{*} \O_X.
\]
The canonical section $1\in f_{*}\O_X$ then yields a map of $ \O_Y$ modules
\[
\tr_f:f_{*}\omega_{X}[d_X] \to \omega_{Y}[d_Y]
\]
which is the trace morphism of coherent sheaf theory. We now use this to construct the analogous map for $\mathcal{D}_h$-modules.

\begin{Lemma}
\label{lem:trace}
If $f:X\to Y$ is a proper morphism then there exists a natural map
\[
\tr:\int_{f}\mathcal{O}_{X,h}[d_{X}]\to\mathcal{O}_{Y,h}[d_{Y}]
\]
of left $\mathcal{D}_{Y,h}$-modules. Applying the functor $\text{gr}$ we obtain a natural map 
\[
\gr(\tr):(p_{2})_{*}(\pi_{f}^{*}i_{*}\mathcal{O}_{X}\otimes p_{1}^{*}\omega_{X/Y})[d_{X}] \longrightarrow i_{*}\mathcal{O}_{Y}[d_{Y}]
\]
where $i$ denotes both inclusions $X\to T^{*}X$ and $Y\to T^{*}Y$. 

Furthermore, we can identify $\gr(\tr)$ with the map 
\[
(p_2)_*(\pi_{f}^{*} i_{*}\mathcal{O}_{X}\otimes p_{1}^{*}\omega_{X/Y}) [d_X] \to (p_2)_*(j_{*}\mathcal{O}_{X} \otimes p_{1}^{*}\omega_{X/Y})[d_X] \to i_{*}\mathcal{O}_{Y}[d_Y]
\]
where the first map is induced by the canonical morphism 
$ \pi_{f}^{*}i_{*}\mathcal{O}_{X} =   \pi_f^* (\pi_f)_* j_* \mathcal{O}_{X} \rightarrow j_* \O_X $ 
(where $ j : X \rightarrow T^*Y \times_Y X $) and the second map is induced by the trace map for $p_{2}$, using that $ j_* \O_X = p_2^*  i_* \O_Y $. 
\end{Lemma}

\begin{proof}
Recall that we have the Spencer resolution 
\[
\text{Sp}(\mathcal{D}_{X,h})\to\mathcal{O}_{X,h}
\]
of $\mathcal{O}_{X,h}$. The terms of this complex are $\mathcal{D}_{X,h}\otimes_{\mathcal{O}_{X,h}}\Theta_{X,h}^{k}\{-2k\}$
(in degree $-k$). From
the definition of push-forward one sees directly that 
\[
\int_{f}\mathcal{D}_{X,h}\otimes_{\mathcal{O}_{X,h}}\Theta_{X,h}^{k}\{-2k\} \cong f_{*}(\mathcal{D}_{Y\leftarrow X,h}\otimes_{\mathcal{O}_{X,h}}\Theta_{X,h}^{k})\{-2k\}
\]
By definition $\mathcal{D}_{Y\leftarrow X,h}=f^{*}\mathcal{D}_{Y,h}\otimes_{\mathcal{O}_{X,h}}f^{*}\omega_{Y,h}^{-1}\otimes_{\mathcal{O}_{X,h}}\omega_{X,h}=f^{*}\mathcal{D}_{Y,h}\otimes_{\mathcal{O}_{X,h}}\omega_{X/Y,h}$.
Thus we have 
\[
f_{*}(\mathcal{D}_{Y\leftarrow X,h}\otimes_{\mathcal{O}_{X,h}}\Theta_{X,h}^{k})\{-2k\} \cong
\mathcal{D}_{Y,h}\otimes_{\mathcal{O}_{Y,h}}f_{*}(\omega_{X/Y,h}\otimes_{\mathcal{O}_{X,h}}\Theta_{X,h}^{k})\{-2k\}
\]
We have isomorphisms 
\begin{align*}
&\mathcal{H}om_{\mathcal{D}_{Y,h}}(\mathcal{D}_{Y,h}\otimes_{\mathcal{O}_{Y,h}}f_{*}(\omega_{X/Y,h}\otimes_{\mathcal{O}_{X,h}}\Theta_{X,h}^{k})\{-2k\}[d_X + k],\mathcal{O}_{Y,h}[d_{Y}])
\\
&\cong \mathcal{H}om_{\mathcal{O}_{Y,h}}(f_{*}(\omega_{X/Y,h}\otimes_{\mathcal{O}_{X,h}}\Theta_{X,h}^{k})\{-2k\}[d_X + k],\mathcal{O}_{Y,h}[d_{Y}])
\\
&\cong f_{*}\mathcal{H}om_{\mathcal{O}_{X,h}}(\omega_{X/Y,h}\otimes_{\mathcal{O}_{X,h}}\Theta_{X,h}^{k}\{-2k\}[d_{X}+k],f^{!}\mathcal{O}_{Y,h}[d_{Y}])
\\
&\cong f_{*}\mathcal{H}om_{\mathcal{O}_{X,h}}(\omega_{X/Y,h}\otimes_{\mathcal{O}_{X,h}}\Theta_{X,h}^{k}\{-2k\}[k],\omega_{X/Y,h}) \cong
f_{*}\Omega_{X,h}^{k}\{2k\}[-k]
\end{align*}
where the third to last isomorphism is Grothendieck duality.
Morover, as the Spencer complex is dual to the de Rham complex, the
induced morphism 
\[
d_{k}:f_{*}\Omega_{X,h}^{k}\{2k\}[-k]\to f_{*}\Omega_{X,h}^{k+1}[-k-1]\{2(k+1)\}
\]
is simply the pushforward, under $f$, of the usual de Rham differential;
i.e., we have an isomorphism 
\[
\mathcal{H}om_{\mathcal{D}_{Y,h}}(\int_{f}\mathcal{O}_{X}[d_{X}],\mathcal{O}_{Y}[d_{Y}])\cong f_{*}\DR(\O_{X,h})
\]
When $k=0$ the section $1\in\mathcal{O}_{X,h}$ yields a canonical
section of $H^{0}(Y,f_{*}\DR(\mathcal{O}_{X,h}))$, and hence a canonical
element of ${\displaystyle \mathcal{H}om_{\mathcal{D}_{Y,h}}(\int_{f}\mathcal{O}_{X,h}[d_{X}],\mathcal{O}_{Y,h}[d_Y])}$,
which we define to be the trace map $ \tr$. 

Let $$ \gr(\tr)  : (p_2)_*( \pi_f^* i_* \O_X \otimes p_1^* \omega_{X/Y})[d_X] \rightarrow i_* \O_Y[d_Y]$$
be the result of applying the functor $\gr$ to the above $ \tr $ and using Corollary \ref{cor:push}.

Consider the following diagram
\begin{equation} \label{eq:doublesquare}
\xymatrix{
(p_2)_*( \pi_f^* i_* \O_X \otimes p_1^* \omega_{X/Y})[d_X] \ar[r]^-{\gr(\tr)} & i_* \O_Y[d_Y] \\
(p_2)_*( \pi_f^* \O_{T^* X} \otimes p_1^* \omega_{X/Y})[d_X] 	\ar[u] \ar[r]^-{\tr_{p_2}} & \O_{T^* Y}[d_Y] \ar[u] \ar@{=}[d]\\
\pi_Y^* f_* \omega_{X/Y} [d_X] \ar[r]^-{\pi_Y^*\tr_f} \ar[u]^\sim & \pi_Y^* \O_Y [d_Y]
}
\end{equation}
where the upper vertical arrows are the natural surjections, the bottom left vertical arrow is the base change isomorphism and the two lower horizontal arrows come from the trace morphism for $\O$-modules. 

The above construction of $ \tr $ implies that the outer square of (\ref{eq:doublesquare}) commutes.  On the other hand, the lower square commutes by the compatibility of Grothendieck duality with (flat) base change
over the square 
\begin{equation*}
\xymatrix{
	\ar[d]^{p_2} X \times_Y T^* Y \ar[rr]^{p_1} & & X \ar[d]^f \\
	T^* Y \ar[rr]^{\pi_Y} & & Y
}
\end{equation*}

Thus, we deduce the commutativity of the upper square of (\ref{eq:doublesquare}), which implies the final claim of the Lemma.

\end{proof}

\begin{Corollary}
Let $M\in D(\mathcal{D}_{X,h}\mod)$. The trace map induces
an isomorphism 
\[
\Phi':f_{*}\mathcal{H}om_{\mathcal{D}_{X,h}}(M,\mathcal{O}_{X,h})[d_{X}] \xrightarrow{\sim} \mathcal{H}om_{\mathcal{D}_{Y,h}}(\int_{f}M,\mathcal{O}_{Y,h})[d_{Y}]
\]
\end{Corollary}

\begin{proof}
In more detail, the morphism is defined by 
\[
f_{*}\mathcal{H}om_{\mathcal{D}_{X,h}}(M,\mathcal{O}_{X,h})[d_{X}]\to\mathcal{H}om_{\mathcal{D}_{Y,h}}(\int_{f}M,\int_{f}\mathcal{O}_{X,h})[d_{X}]\to\mathcal{H}om_{\mathcal{D}_{Y,h}}(\int_{f}M,\mathcal{O}_{Y,h})[d_{Y}]
\]
where the last map is induced by the trace.  By passing to an open affine, it suffices to check the corresponding statement where $ \mathcal Hom$ is replaced by $\Hom $.

By the graded Nakayama lemma, it suffices to prove that this map is an isomorphism after
applying $\text{gr}$; if we let $M_{0}=\text{gr}(M)$ then we are
considering the chain of maps 
\begin{equation} \label{eq:wantIso}
\begin{gathered}
\Hom_{\mathcal{O}_{T^{*}X}}(M_{0},i_{*}\mathcal{O}_{X})[d_{X}]\to \Hom_{\mathcal{O}_{T^{*}Y}}((p_2)_*(\pi_{f}^{*}M_{0}\otimes p_{1}^{*}\omega_{X/Y}),(p_2)_*(\pi_{f}^{*}i_{*}\mathcal{O}_{X}\otimes p_{1}^{*}\omega_{X/Y}))[d_{X}]
\\
\longrightarrow \Hom_{\mathcal{O}_{T^{*}Y}}((p_2)_*(\pi_{f}^{*}M_{0}\otimes p_{1}^{*}\omega_{X/Y}),i_{*} \mathcal{O}_{Y})[d_{Y}]
\end{gathered}
\end{equation}
where the second map is obtained by applying $\gr(\tr)$. 

Using that fact that $\pi_{f} \circ j = i$ and adjunction for the map $\pi_f$, we have that the map 
\begin{equation} \label{eq:haveIso1}
\Hom_{\mathcal{O}_{T^{*}X}}(M_{0},i_{*}\mathcal{O}_{X})\to\Hom_{\mathcal{O}_{X\times_{Y}T^{*}Y}}(\pi_{f}^{*}M_{0},\pi_{f}^{*}i_{*}\mathcal{O}_{X})
\to\Hom_{\mathcal{O}_{X\times_{Y}T^{*}Y}}(\pi_{f}^{*}M_{0},j_{*} \mathcal{O}_{X})
\end{equation}
is an isomorphism.  By Grothendieck duality, the following composition is an isomorphism
\begin{equation} \label{eq:haveIso2}
\begin{gathered}
\Hom_{\mathcal{O}_{X\times_{Y}T^{*}Y}}(\pi_{f}^{*}M_{0},j_{*}\mathcal{O}_{X})[d_X]
\rightarrow \Hom_{\mathcal{O}_{T^{*}Y}}((p_2)_*(\pi_{f}^{*}M_{0}\otimes p_{1}^{*}\omega_{X/Y}),(p_2)_*(j_* \mathcal{O}_{X} \otimes p_{1}^{*}\omega_{X/Y}))[d_X] \\
\rightarrow \Hom_{\mathcal{O}_{T^{*}Y}}((p_2)_*(\pi_{f}^{*}M_{0}\otimes p_{1}^{*}\omega_{X/Y}),i_{*}\mathcal{O}_{Y})[d_Y]
\end{gathered}
\end{equation}
 On the other hand, by the lemma above, $\gr(\tr)$
agrees with the composition 
\[
(p_2)_*(\pi_{f}^{*}i_{*}\mathcal{O}_{X}\otimes p_{1}^{*}\omega_{X/Y})[d_X] \to (p_2)_*(j_{*}\mathcal{O}_{X} \otimes p_{1}^{*}\omega_{X/Y})[d_X] \to i_{*}\mathcal{O}_{Y}[d_Y]
\]
This implies that the map (\ref{eq:wantIso}) is the composition of (\ref{eq:haveIso1}) and (\ref{eq:haveIso2}).  Since these latter two maps are isomorphisms, we deduce that (\ref{eq:wantIso}) is an isomorphism, completing the proof. 
\end{proof}

\begin{proof}[Proof of Theorem \ref{prop:duality}]

By the projection formula, Proposition \ref{prop:projection}, we have 
	\[
	\int_{f}f^{*}\mathcal{D}_{Y,h} \cong  \bigl(\int_{f}\O_{X,h}\bigr)\otimes_{\O_{Y,h}}\mathcal{D}_{Y,h}
	\]		
	Thus the trace map induces a map 
	\begin{equation}
	\label{eq:intfToD}
	\int_{f} f^* \D_{Y,h} [d_X]\to \O_{Y,h}[d_Y] \otimes_{\O_{Y,h}}\mathcal{D}_{Y,h} \cong \mathcal{D}_{Y,h}[d_Y]
	\end{equation}
	Now let $M\in D(\mathcal{D}_{X,h}\mod)$.
	Then 
	\[
	\int_{f}\mathbb{D}_{X}M \cong f_{*}(\mathcal{H}om_{\mathcal{D}_{X,h}}(M,f^{*}\mathcal{D}_{Y,h}))\otimes_{\O_{Y,h}}\omega_{Y,h}^{-1}[d_X]
	\]
	while 
	\[
	\mathbb{D}_{Y}\int_{f}M \cong \mathcal{H}om_{\mathcal{D}_{Y,h}}(\int_{f}M,\mathcal{D}_{Y,h})\otimes_{\O_{Y,h}}\omega_{Y,h}^{-1}[d_Y]
	\]
Thus, by Remark \ref{re:PushFunctoriality} and (\ref{eq:intfToD}), we have morphisms
\begin{align*}
	\int_{f}\mathbb{D}_{X}M &\cong f_* \mathcal Hom_{\D_{X,h}}(M, f^* \D_{Y,h}[d_X]) \otimes \omega_{Y,h}^{-1} \\
	&\rightarrow \mathcal Hom_{\D_{Y,h}}(\int_f M, \int_f f^* \D_{Y,h}[d_X]) \otimes \omega_{Y,h}^{-1}\\
	&\rightarrow \mathcal Hom_{\D_{Y,h}}(\int_f M, \D_{Y,h}[d_Y]) \otimes \omega_{Y,h}^{-1} \cong \mathbb D_Y \int_f M
\end{align*}
We would like to prove that the resulting composition is an isomorphism.  To the map 
\[
\Phi: f_{*}(\mathcal{H}om_{\mathcal{D}_{X,h}}(M,f^{*}\mathcal{D}_{Y,h})[d_X] \to \mathcal{H}om_{\mathcal{D}_{Y,h}}(\int_{f}M,\mathcal{D}_{Y,h})[d_{Y}]
\]
we may apply the functor $\otimes_{\mathcal{D}_{Y,h}} \mathcal{O}_{Y,h}$. One verifies easily that the result is the isomorphism
\[
\Phi':f_{*}\mathcal{H}om_{\mathcal{D}_{X,h}}(M,\mathcal{O}_{X,h})[d_{X}] \xrightarrow{\sim}
\mathcal{H}om_{\mathcal{D}_{Y,h}}(\int_{f}M,\mathcal{O}_{Y,h})[d_{Y}]
\]
of the previous corollary. Since $\mathcal{O}_{Y,h} \cong \mathcal{D}_{Y,h}/I$, where $I$ is an ideal sheaf consisting of positively graded elements, the (proof of) the graded Nakayama lemma shows that $\Phi$ must be an isomorphism as well.  
\end{proof}

\begin{Corollary}\label{cor:D}
We have $\bD_X \O_{X,h} \cong \O_{X,h} \{2d_X\}$. Moreover, for a closed subvariety $X \subset Y$ we have $ \bD_Y (\delta_{X,h}) = \delta_{X,h}$.
\end{Corollary}
\begin{proof}
Using the Spencer resolution $ \Sp(\D_{X,h}) \rightarrow \O_{X,h} $ we have
$$ \bD_X \Sp(\D_{X,h}) = \Hom_{\D_{X,h}}( \Sp(\D_{X,h}), \D_{X,h}) \otimes_{\O_{X,h}} \omega_{X,h}^{-1} [d_X]. $$
Since $ \Sp(\D_{X,h}) $ is a complex of free $ \D_{X,h} $-modules, we find that $ \Hom(\Sp(\D_{X,h}), \D_{X,h}) $ is the complex of right $ \D_{X,h}$-modules
$$
\D_{X,h} \rightarrow \Omega^1_{X,h} \otimes_{\O_{X,h}} \D_{X,h} \{2\} \rightarrow \cdots \rightarrow \Omega^{d_X}_{X,h} \otimes_{\O_{X,h}} \D_{X,h} \{2d_X\}
$$
Then applying $ \otimes_{\O_{X,h}} \omega_{X,h}^{-1}[d_X] $ and using $ \Omega_{X,h}^i \otimes_{\O_{X,h}} \omega_{X,h}^{-1} = \Theta_{X,h}^{d_X - i} $  we deduce that $ \bD_X \Sp(\D_{X,h}) = \Sp(\D_{X,h})\{2d_X\}$.

The second assertion the follows from the first together with Proposition \ref{prop:duality}.
\end{proof}

\begin{Proposition} \label{pr:adj1}
Let $ f : X \rightarrow Y $ be a proper morphism.  Let $ M \in D(\D_{X,h}\mod) $ and $ N \in D(\D_{Y,h}\mod)$.  We have a natural isomorphism
$$
\Hom_{\D_{Y,h}}(\int_f M, N) \cong \Hom_{\D_{X,h}}(M, f^\dagger N)
$$
\end{Proposition}
\begin{proof} 
This  follows in a completely formal way from the isomorphism of Proposition \ref{prop:duality} -- cf. the proof of \cite[Corollary 2.7.3]{HTT}. 
\end{proof}

To end this section, we would like to record a compatibility of the Verdier duality theorem (Proposition \ref{pr:adj1}) proved above with the Grothendieck duality theorem for $\mathcal{O}$-modules, which essentially follows from our method of construction.

Now,  let $\mathcal{F}$ be a coherent sheaf on $X\times \mathbb{A}^1$. Following M. Saito \cite{Sa3}, we denote by $\Ind \mathcal{F} $ the coherent left $\mathcal{D}_{X,h}$-module $(\mathcal{F} \otimes_{\O_{X,h}} \mathcal{D}_{X,h}) \otimes \omega_{X,h}^{-1}$; this operation defines an exact functor $\Ind$ from coherent sheaves on $X\times \mathbb{A}^1$ to $\mathcal{D}_{X,h}$-modules. 

By adjuction for tensor products we have
\begin{equation} \label{eq:adjInd}
\mathcal{H}om_{\mathcal{D}_{X,h}}(\Ind\mathcal{F},N) \xrightarrow{\sim} \mathcal{H}om_{\mathcal{O}_{X,h}}(\mathcal{F},N\otimes \omega_{X,h}) 
\end{equation}

If $f:X\to Y$ is a proper morphism then a quick computation shows that
\[
\int_{f} \Ind \mathcal{F}  \cong \Ind f_{*}\mathcal{F}
\]

where on the right hand side we have implicitly derived the (exact) functor $\Ind$.

We now show that the adjuction morphisms for $\D_h$-modules are compatible with the ones for $\mathcal{O}$-modules under the induction functor and the above isomorphism. 

\begin{Proposition}
Let $f:X \to Y$ be proper, and let $\mathcal{F}$ be a coherent sheaf on $X$. Then the following diagram commutes:
$$\xymatrix{
\Hom_{\mathcal{D}_{X,h}}(\Ind\mathcal{F},f^{\dagger}N) \ar[r] \ar[d] & \Hom_{\mathcal{O}_{X,h}}(\mathcal{F},f^{*}N\otimes \omega_{X,h})[d_X-d_Y] \ar[d] \\
\Hom_{\mathcal{D}_{Y,h}}(\int_f\Ind\mathcal{F},N) \ar[r] & \Hom_{\mathcal{O}_{Y,h}}(f_{*}\mathcal{F},N \otimes \omega_{Y,h})
}$$
here, the vertical arrows are given by Verdier and Grothendieck duality, respectively, and the horizontal arrows are given by adjunction for tensor products. 
\end{Proposition}

\begin{proof}
We consider the diagram
\begin{equation} \label{eq:PentSquareTri}
\xymatrix{
\Hom_{\D_{X,h}}(\Ind \mathcal F, f^*N) \ar[rr] \ar[d] & & \Hom_{\O_{X,h}}(\mathcal F, f^*N \otimes \omega_{X,h}) \ar[d] \\
\Hom_{\D_{Y,h}}(\int_f \Ind \mathcal F, N \otimes \int_f \O_{X,h}) \ar[r] \ar[d] & \Hom_{\O_{Y,h}}(f_* \mathcal F, N \otimes \int_f \O_{X,h} \otimes \omega_{Y,h}) \ar[d] & \Hom_{\O_{Y,h}}(\mathcal F, N \otimes f_* \omega_{X,h}) \ar[l] \ar[dl] \\
\Hom_{\D_{Y,h}}(\int_f \Ind \mathcal F, N)[d_Y - d_X] \ar[r] & \Hom_{\O_{Y,h}}(\mathcal F, N \otimes \omega_{Y,h})[d_Y - d_X] & 
}
\end{equation}	
Here the rightward pointing horizonal morphisms come from the adjunction between induction and restriction (\ref{eq:adjInd}), the upper vertical arrows come from the functoriality of push forward and the projection formulae (for both $ \D_h $ and $ \O_h$-modules), and the lower vertical arrows come from the trace morphism constructed in Lemma \ref{lem:trace}.  The leftward pointing horizonal arrow comes from the morphism of $ \O_{X,h}$-modules, $ \omega_{X/Y,h} \rightarrow \D_{Y \leftarrow X, h} $ and using that $\int_f \O_{X,h} = f_*  \D_{Y \leftarrow X, h} $, while the diagonal arrow comes from the trace morphism for $ \O$-modules.

The upper pentagon of (\ref{eq:PentSquareTri}) commutes by unpacking the definitions.  The bottom left square of (\ref{eq:PentSquareTri}) commutes by the naturality of the adjunction (\ref{eq:adjInd}).  

Finally we consider the bottom right triangle of (\ref{eq:PentSquareTri}).  From Lemma \ref{lem:trace}, the trace morphism is characterized by the commutativity of the square
$$
\xymatrix{
	\D_{Y,h} \otimes f_* \omega_{X,h}[d_X] \ar[r] \ar[d] & \int_f \O_{X,h} \otimes \omega_{Y,h}[d_X] \ar[d] \\
	\D_{Y,h} \otimes \omega_{Y,h}[d_Y] \ar[r] & \O_{Y,h} \otimes \omega_{Y,h}[d_Y]
}
$$
where the left vertical arrow is the trace morphism for $ \O$-modules and the right vertical arrow is the one for $ \D_h$-modules.  Mapping $ f_* \omega_{X,h} $ into the top left corner implies the commutativity of the bottom right triangle of (\ref{eq:PentSquareTri}).

Thus we conclude that (\ref{eq:PentSquareTri}) commutes, which implies the desired statement after shifting by $ d_X - d_Y $.

\end{proof}

\begin{Corollary} \label{cor:explicit-adjunction}
Let $f:X \to Y$ be a smooth immersion, and let $M$ be a coherent $\mathcal{D}_{X,h}$-module and $N$ a coherent $\mathcal{D}_{Y,h}$-module. Suppose that the canonical bundles of $X$ and $Y$ are trivial, and that $X$ is defined by an ideal sheaf $\mathcal{I}$. In this case, a map $\varphi: M \to f^{\dagger}N$ is simply a map $\varphi: M \to H^0(f^\d N) = \operatorname{ker}(\mathcal{I})\subset N$. Then the adjoint map 
\[
\int_{f}M \cong f^*\mathcal{D}_{Y,h} \otimes_{\mathcal{D}_{X,h}} M \to N
\]
is given by $P \otimes m \to P \cdot \varphi(m)$ for any section $P \in f^*\mathcal{D}_{Y,h}$. 
\end{Corollary}

\begin{proof}
If $M$ is an induced module, this follows immediately from the previous proposition. In general, any coherent $\mathcal{D}_{X,h}$-module $M$ is a quotient of an induced module, since by the definition of coherence, there is some coherent $\mathcal{O}$ sub-module $\mathcal{F} \subset M$ such that $\mathcal{D}_{X,h} \cdot \mathcal{F} = M$; which implies that the induced map $\Ind\mathcal{F} \to M$ is onto. But then the corollary for $M$ follows from the corollary for $\Ind\mathcal{F}$ by considering the composed map 
\[
\Ind\mathcal{F} \to M \to f^{\dagger}N
\]
since the functor $\int_{f}$ is exact. 
\end{proof}

\begin{Corollary} \label{cor:adjuction-for-induced}
Let $f:X \to Y$ be a smooth immersion, let $\mathcal{F}$ be a coherent sheaf on $X \times \mathbb{A}^1$ and let $M=\Ind \mathcal{F}$. Then the adjunction $M \to H^{0} (f^{\dagger} \int_{f}M)$ is an inclusion.  More precisely, after localizing so that the canonical bundles of $X$ and $Y$ are trivial, it is the map 
\[
\mathcal{F} \otimes \mathcal{D}_{X,h} \to \mathcal{F} \otimes f^*\mathcal{D}_{Y,h} 
\]
which takes $s \otimes P$ to $s \otimes P\cdot1$ (where $P\cdot1$ denotes the action of $P \in \mathcal{D}_{X,h}$ on $1\in f^*\mathcal{D}_{Y,h}$). 
\end{Corollary}

\begin{proof}
Consider the commutative diagram of the previous proposition where $N= \int_{f}M$:
$$\xymatrix{
\mathcal{H}om_{\mathcal{D}_{X,h}}(\Ind\mathcal{F},f^{\dagger} \int_{f}\Ind\mathcal{F}) \ar[r] \ar[d] & \mathcal{H}om_{\mathcal{O}_{X,h}}(\mathcal{F},f^{*} \int_{f} \Ind\mathcal{F}\otimes \omega_{X,h})[d_X-d_Y] \ar[d] \\
\mathcal{H}om_{\mathcal{D}_{Y,h}}(\int_f\Ind\mathcal{F},\int_f\Ind\mathcal{F}) \ar[r] & \mathcal{H}om_{\mathcal{O}_{Y,h}}(f_{*}\mathcal{F},(\Ind f_* \mathcal{F}) \otimes \omega_{Y,h})
}$$
The morphism in question is the unique map in the top left corner which maps to the identity morphism in the bottom left corner. Moving over to the bottom right corner; since 
\[
(\Ind f_*\mathcal{F}) \otimes \omega_{Y,h} \cong f_*\mathcal{F} \otimes_{\mathcal{O}_{Y,h}} \mathcal{D}_{Y,h}
\]
one sees directly that the identity map corresponds to the inclusion $f_*\mathcal{F} \to f_*\mathcal{F} \otimes_{\mathcal{O}_{Y,h}} \mathcal{D}_{Y,h}$ taking a section $s \to s \otimes 1$.  

To compute the corresponding map in the top right corner, we need a bit of Grothendieck duality for a closed immersion (as in \cite[section 3.7]{HartRD}). Localizing if necessary, we may assume that the ideal of $X$ in $Y$ is given by a length $ r $ regular sequence (where $ r = d_Y - d_X$) and that the canonical bundles of $X$ and $Y$ are trivial.

Then we have an isomorphism 
 $$H^r(f^{*}(f_*\mathcal{F} \otimes_{\mathcal{O}_{Y,h}} \mathcal{D}_{Y,h})) \cong \mathcal F \otimes f^* \D_{Y,h}$$
 and the map in the top right corner is the natural inclusion $ \mathcal F \rightarrow \mathcal F 
\otimes f^* \D_{Y,h} $.  Passing to the top left using the adjunction (\ref{eq:adjInd}) gives the desired result.
\end{proof}

\subsection{Adjunction for a smooth morphism} \label{sec:adjoints}
Our goal in this section is to prove 

\begin{Proposition} \label{pr:adj2}
Let $ f : X \rightarrow Y $ be a smooth and proper morphism.  Let $ M \in D(\D_{X,h}\mod) $ and $ N \in D(\D_{Y,h}\mod)$.  We have a natural isomorphism
$$ \Hom_{\D_{X,h}}(f^* N [-(d_X-d_Y)]\{2(d_X-d_Y)\}, M) \cong \Hom_{\D_{Y,h}}(N, \int_f M). $$
\end{Proposition}

We will prove this in Lemmas \ref{lem:1} and \ref{lem:2} below. Our argument is an adaptation of a well-known strategy for $\mathcal{D}$-modules.

As in section \ref{se:notation}, the \emph{abelian} category of quasi-coherent $\mathcal{D}_{X,h}$-modules is identified with the \emph{abelian} category of quasi-coherent $ \O_{X,h}$-modules with flat $h$-connection.

In this language, we may define a natural left exact functor 
\[
f_{\star}:\D_{X,h}\mod \to \D_{Y,h}\mod
\]
as follows. If $\sF$ is an $ \O_{X,h}$-module with flat $h$-connection we denote by $\nabla_{X/Y}$ the composition
\[
 \sF \to  \Omega_{X,h}^{1} \{2\}  \otimes_{\O_{X,h}} \sF \to \Omega_{X/Y,h}^{1} \{2\} \otimes_{\O_{X,h}} \sF .
\]
This gives $\mathcal{F}$ the structure of an $h$-connection over
$Y$, meaning that the map $\nabla_{X/Y}$ satisfies the $h$-Leibniz rule
and the flatness condition. 

Consider the subsheaf $\mathcal{F}^{\nabla_{X/Y}}\subset\mathcal{F}$
of local sections which are annihilated by $\nabla_{X/Y}$. Restricting
$\nabla$ to this subsheaf therefore yields 
\[
\nabla:\sF^{\nabla_{X/Y}}\to  f^{*}\Omega_{Y,h}^{1} \{2\} \otimes_{\O_{Y,h}} \sF^{\nabla_{X/Y}}
\]
Applying the sheaf push-forward yields a map 
\[
\nabla_{Y}:f_{*}(\mathcal{F}^{\nabla_{X/Y}})\to \Omega_{Y,h}^{1} \{2\} \otimes_{\mathcal{O}_{X,h}} f_{*}(\mathcal{F}^{\nabla_{X/Y}})
\]
which is easily seen to be a flat $h$-connection on $f_{*}(\mathcal{F}^{\nabla_{X/Y}})$. This defines the functor $f_{\star}$. 

It will also be useful to write this functor in the language of $\mathcal{D}_h$-modules. For this we note that, by an easy local computation, the functor 
\[
\mathcal{F} \to \mathcal{F}^{\nabla_{X/Y}}
\]
is isomorphic to the functor 
\[
\mathcal{F} \to \mathcal{H}om_{\mathcal{D}_{X,h}}(f^{*}\mathcal{D}_{Y,h},\mathcal{F})
\]

This implies that the sheaf $\mathcal{F}^{\nabla_{X/Y}}$ admits the structure of a sheaf of $f^{-1}\mathcal{D}_{Y,h}$-modules, and if $\mathcal{F}$ is a coherent $\mathcal{D}_{X,h}$-module, then $\mathcal{F}^{\nabla_{X/Y}}$ is a coherent (i.e., locally finitely presented) sheaf of $f^{-1}\mathcal{D}_{Y,h}$-modules. The same is true for the corresponding derived functor. 

Now suppose that $f$ is also proper. Since the functor $f_{\star}$ is the composition of $\mathcal{F} \to \mathcal{F}^{\nabla_{X/Y}}$  with $f_*$, and since the derived functor of $f_{*}$ takes the bounded derived category of coherent $f^{-1}\mathcal{D}_{Y,h}$-modules to the bounded derived category of $\mathcal{D}_{Y,h}$-modules, we see that there is a derived functor of $f_{\star}$, which we also denote 
\[
f_{\star}: D(\mathcal{D}_{X,h}-mod) \to D(\mathcal{D}_{Y,h}-mod)
\]
which is simply the composition of the derived functor $\mathcal{F} \to \mathcal{H}om_{\mathcal{D}_{X,h}}(f^{*}\mathcal{D}_{Y,h},\mathcal{F})$ with the derived functor of $f_{*}$. 


\begin{Lemma}\label{lem:1}
\label{lem:two-pushforwards} 
For a smooth, proper morphism $f: X \to Y$ as above, the functor
\[
f_{\star}:D(\text{qcoh}(\mathcal{D}_{X,h}))\to D(\text{qcoh}(\mathcal{D}_{Y,h}))
\]
is isomorphic to $\int_f [-(d_X-d_Y)] \{2(d_X-d_Y)\}$. 
\end{Lemma}

\begin{Remark}
If one is willing to work with the derived categories of quasi-coherent sheaves, then one can remove the condition that $f$ be proper in this lemma. 
\end{Remark}

\begin{proof}
By Lemma \ref{le:deRham}, we have 
a free resolution of $\mathcal{D}_{Y\leftarrow X,h}$ in the category
of $\mathcal{D}^{\mathrm{op}}_{X,h}$ modules. Thus for any $M\in D(\text{qcoh}(\mathcal{D}_{X,h}))$
we have isomorphisms 
\[
\mathcal{D}_{Y\leftarrow X,h}\otimes_{\mathcal{D}_{X,h}}M \cong \DR_{Y}(\mathcal{D}_{X,h})\otimes_{\mathcal{D}_{X,h}}M[d_X - d_Y]\{-2(d_Y - d_X)\}
\]
\[
 \cong \mathcal{H}om_{\mathcal{D}_{X,h}}(\Sp_{Y}(\mathcal{D}_{X,h}),M)[d_X - d_Y]\{-2(d_X - d_Y)\}
\]
in $D(\text{qcoh}(\mathcal{D}_{X,h}))$. Here we have used the isomorphisms
\[
\mathcal{H}om_{\mathcal{D}_{X,h}}(\mathcal{D}_{X,h}\otimes_{\mathcal{O}_{X,h}}\Theta_{X/Y,h}^{i},M) \cong \Omega_{X/Y,h}^{i}\otimes_{\mathcal{O}_{X,h}} M \cong \DR_{Y}^{i}(\mathcal{D}_{X,h})\otimes_{\mathcal{D}_{X,h}}M.
\]
Now by Lemma \ref{lem:Spencer}, $\Sp_{Y}(\mathcal{D}_{X,h})$ is a resolution of $f^{*}\mathcal{D}_{Y,h}$
which by modules which are locally free in the category of $\mathcal{D}_{X,h}$-modules.

Therefore, if $M\in D(\text{qcoh}(\mathcal{D}_{X,h}))$, we have 
\[
\mathcal{H}om_{\mathcal{D}_{X,h}}(\Sp_{Y}(\mathcal{D}_{X,h}),M) \cong \mathcal{H}om_{\mathcal{D}_{X,h}}(f^{*}\mathcal{D}_{Y,h},M)
\]
Thus, combining these isomorphisms, we conclude that
\begin{equation}
\label{eq:intHom}
\int_f M = f_* \mathcal Hom_{\D_{X,h}}(f^* \D_{Y,h}, M)[d_X - d_Y]\{-2(d_X - d_Y)\}
\end{equation}

Finally, we recall that, as discussed above, on the \emph{abelian} category of quasi-coherent $\mathcal{D}_{X,h}$-modules, the functor $M\to\mathcal{H}om_{\mathcal{D}_{X,h}}(f^{*}\mathcal{D}_{Y,h},M)$ is the functor taking $M$ to its subsheaf of local sections which are annihilated by the action of $\Theta_{X/Y}$ (which is precisely $M^{\nabla_{X/Y}}$). Therefore the functor $f_{\star}$ is the composition of the left exact functors $f_{*}$ and $\mathcal{H}om_{\mathcal{D}_{X,h}}(f^{*}\mathcal{D}_{Y,h},M)$, hence its derived functor is given by the composition
\[
M\to f_{*}\mathcal{H}om_{\mathcal{D}_{X,h}}(f^{*}\mathcal{D}_{Y,h},M)
\]
where $f_{*}$ and $\mathcal{H}om$ now denote the derived functors. Combining with (\ref{eq:intHom}) gives the desired result.
\end{proof}

\begin{Lemma}\label{lem:2}
\label{lem:Naive-adjunction}The functor $f_{\star}$ is right adjoint
to $f^{*}$. 
\end{Lemma}
\begin{proof}
The functor $f^{*}$ is defined as
the composition of $f^{-1}$ and $f^{*}D_{Y,h}\otimes_{f^{-1} \mathcal{D}_{Y,h}}$.
In the course of the previous proof, we have seen that the functor
$f_{\star}$ is the composition of $\mathcal{H}om_{\mathcal{D}_{X,h}}(f^{*}\mathcal{D}_{Y,h},)$
and $f_{*}$. So the result follows from the general fact that $(f^{-1},f_{*})$
and 
\[
(f^{*}D_{Y,h}\otimes_{f^{-1} \mathcal{D}_{Y,h}},\mathcal{H}om_{\mathcal{D}_{X,h}}(f^{*}\mathcal{D}_{Y,h},))
\]
are adjoint pairs of functors. 
\end{proof}

\subsection{Base change}


\begin{Proposition}\label{prop:basechange}
Consider a fiber product diagram 
$$\xymatrix{
Y \times_X Z \ar[r]^{\tf} \ar[d]^{\tg} & Z \ar[d]^{g} \\
Y \ar[r]^f & X
}$$
where all varieties are smooth, $f,g$ are proper maps and $Y \times_X Z$ is of the expected dimension. We have an isomorphism $\int_{\tf} \tg^\d \cong g^\d \int_f$ of functors $ D(\D_{Y,h}\mod) \to D(\D_{Z,h}\mod)$. 
\end{Proposition}

\begin{proof}
	We will prove that the following composition of adjunctions
	\begin{equation}\label{eq:naturalmap}
	\int_{\tf} \tg^\d \to \int_{\tf} \tg^\d f^\d \int_f =  \int_{\tf}  \tf^\d g^\d \int_f \to g^\d \int_f
	\end{equation}
	is an isomorphism. 
	
	We can factor $f$ as $Y \xrightarrow{id \times f} Y \times X \to X$ where the second map is the projection. Subsequently we can assume that $f$ is either a projection or an embedding (and likewise for $g$). The case of the projection follows easily as in the proof of the base change formula for $\D$-modules in \cite[Theorem 1.7.3]{HTT}. So we will assume from now on that both $f$ and $g$ are embeddings. 
	
	We can also assume that all the varieties are affine.  Moreover assume that $Y $ is a divisor defined by a function $ t \in \O(X) $ and that $ Z $ is codimension $ m $ (we can simply iterate to obtain the case where $ Y $ is of higher codimension).  Choose a vector field $ \partial_t $ on $ X $ such that $ \partial_t(t) = 1 $.
	
	Finally, because the varieties are affine, it suffices to check the isomorphism for the object $ \D_{Y,h} $.
	
	We have
	$$
	\int_f \D_{Y,h} = \D_{Y,h} \otimes_\C \C[\partial_t]
	$$
	and so 
	$$
	f^\d \int_f \D_{Y,h} =  \D_{Y,h} \otimes_\C \C[\partial_t] \xrightarrow{t} \D_{Y,h} \otimes_\C \C[\partial_t]
	$$
	where the complex is placed in degrees $ 0 $ and $ 1 $.
	
	By Corollary \ref{cor:adjuction-for-induced}, the adjuction map $ \D_{Y,h} \rightarrow f^\d \int_f \D_{Y,h}  $ is given by the natural map \begin{equation}
	\label{eq:firstfirstAdjunction}
	 \D_{Y,h} \rightarrow \D_{Y,h} \otimes_\C \C[\partial_t]
	 \end{equation}
	
	On the other hand, the pullbacks to $ Z $ and $Y \cap Z $ are exact and we have 
	$$
	g^\d \int_f \D_{Y,h} = \O_{Z,h} \otimes_{\O_{X,h}} \D_{Y,h} \otimes_\C \C[\partial_t] [-m]
	$$
	and
	$$ \tilde g^\d \D_{Y,h} = \O_{Y \cap Z,h} \otimes_{\O_{Y,h}} \D_{Y,h} [-m]
	$$
	
	Also,
	$$
	\tilde g^\d f^\d \int_f \D_{Y,h} =  \O_{Y \cap Z,h} \otimes_{\O_{Y,h}} (\D_{Y,h} \otimes_\C \C[\partial_t] \xrightarrow{t} \D_{Y,h} \otimes_\C \C[\partial_t]) [-m]
	$$
	
	From (\ref{eq:firstfirstAdjunction}), the adjuction map $ \tilde g^\d \D_{Y,h} \rightarrow \tilde g^\d f^\d \int_f \D_{Y,h} = \tilde f^\d g^\d \int_f \D_{Y,h} $ becomes the inclusion 
	\begin{equation}
	\label{eq:firstAdjunction}
	\O_{Y \cap Z,h} \otimes_{\O_{Y,h}} \D_{Y,h} [-m] \rightarrow  \O_{Y \cap Z,h} \otimes_{\O_{Y,h}} \D_{Y,h} \otimes_\C \C[\partial_t] [-m]
	\end{equation}
	
	The natural isomorphism $ \tilde g^\d f^\d \cong \tilde f^\d g^\d $ is given by the isomorphism $\O_{Y \cap Z,h} \otimes_{\O_{Y,h}} \D_{Y,h} \cong \O_{Z,h} \otimes_{\O_{X,h}} \D_{Y,h} $ and we have
	$$
	\tilde f^\d g^\d \int_f \D_{Y,h} = 
	(\O_{Z,h} \otimes_{\O_{X,h}} \D_{Y,h} \otimes_\C \C[\partial_t] \xrightarrow{t} \O_{Z,h} \otimes_{\O_{X,h}}  \D_{Y,h} \otimes_\C \C[\partial_t]) [-m]
	$$
	
	By Corollary \ref{cor:explicit-adjunction}, the adjunction map 
	$ \int_{\tilde f} \tilde f^\d g^\d \int_f \D_{Y,h} \rightarrow g^\d \int_f \D_{Y,h} $
	is given by
	\begin{equation}
	\label{eq:secondAdj}
	\int_{\tilde f} (\O_{Z,h} \otimes_{\O_{X,h}} \D_{Y,h} \otimes_\C \C[\partial_t] \xrightarrow{t} \O_{Z,h} \otimes_{\O_{X,h}}  \D_{Y,h} \otimes_\C \C[\partial_t]) [-m]\rightarrow \O_{Z,h} \otimes_{\O_{X,h}} \D_{Y,h} \otimes_\C \C[\partial_t] [-m]
	\end{equation}
	(projection out of the first term in the complex on the left).
	
	Composing (\ref{eq:firstAdjunction}) and (\ref{eq:secondAdj}), we see that the resulting map  $ \int_{\tilde f} \tilde g^\d \D_{Y,h} \rightarrow g^\d \int_f \D_{Y,h}$ becomes the natural isomorphism
	$$
	\int_{\tilde f} \O_{Y \cap Z,h} \otimes_{\O_{Y,h}} \D_{Y,h}  [-m] \rightarrow \O_{Y \cap Z,h} \otimes_{\O_{Y,h}} \D_{Y,h} \otimes_\C \C[\partial_t]  [-m] 
	$$
	
\end{proof}

\begin{Corollary} \label{cor:Deltas}
	Suppose that $Y$ and $Z$ are two smooth subvarieties of $X$ which intersect transversally.
	Then we have $\delta_{Y,h} \dotimes_{\O_{X,h}} \delta_{Z,h} \cong \delta_{Y \cap Z,h} \la - d_X \ra$.
\end{Corollary}

\begin{proof}
	Consider the diagram
	$$
	\xymatrix{
		Y \cap Z \ar^{\tilde \Delta}[r] \ar^{\tilde i}[d] & Y \times Z \ar^i[d] \\
		X \ar^{\Delta}[r] & X \times X
	}
$$
Then we have 
\begin{align*}
\delta_{Y,h} \dotimes_{\O_{X,h}} \delta_{Z,h} &= \Delta^\d \int_i \O_{Y \times Z,h} \{ d_Y + d_Z\} \\
&\cong \int_{\tilde i} \tilde \Delta^\d \O_{Y \times Z, h} \{d_Y + d_Z \} \\
&\cong \int_{\tilde i} \O_{Y \cap Z, h}[d_{Y \cap Z} -d_Y - d_Z] \{d_Y + d_Z \} 
\end{align*}
where in the second line we use Proposition \ref{prop:basechange}.  The result follows since $ d_{Y \cap Z} = d_Y + d_Z - d_X $.
\end{proof}
\begin{Corollary}\label{cor:pullbackdelta}
	Suppose $f: X \rightarrow Y$ is a morphism between smooth varieties and $Z \subset Y$ a smooth, closed subvariety. If $f^{-1}(Z) \subset X$ is smooth of the expected dimension then
	$$f^\d(\delta_{Z,h}) \cong \delta_{f^{-1}(Z),h} \la d_X - d_Y \ra .$$
\end{Corollary}
\begin{proof}
	The expected dimension hypothesis implies that $f^* \O_{Z,h} \cong \O_{f^{-1}(Z),h}$ (apply Corollary \ref{cor:Deltas} to the graph of $f$ and $X \times f^{-1}(Z)$ inside $X \times Y$). So we have
	\begin{align*}
	f^\d(\delta_{Z,h})
	&\cong f^* \O_{Z,h} [d_X - d_Y] \{d_Z\} \\
	&\cong \O_{f^{-1}(Z),h}\{\dim(f^{-1}(Z))\} [d_X-d_Y] \{-d_X+d_Y\} \\
	&\cong \delta_{f^{-1}(Z),h} \la d_X - d_Y \ra
	\end{align*}
	where we used that $\dim(f^{-1}(Z)) = d_X-d_Y+d_Z$.
\end{proof}

\section{Kernels and the associated graded functor}\label{sec:grkernels}

In this section we recall and develop the theory of kernels for $\O_X$ and $\D_{X,h}$ modules.

\subsection{Kernels for $\O_X$-modules}\label{sec:kerO}

First recall the formalism of kernels for $\O_X$-modules. Let $X, Y$ be two smooth varieties. A kernel is any object $\sP \in D(\O_{X \times Y}\mod)$. It defines the associated integral (or Fourier-Mukai) transform
\begin{equation*}
\begin{aligned}
\Phi_\sP : D(\O_X\mod) &\rightarrow D(\O_Y\mod) \\ M &\mapsto {\pi_2}_* (\pi_1^* (M) \otimes \sP)
\end{aligned}
\end{equation*}
where $\pi_1$ and $\pi_2$ are the projections from $X \times Y$ to $X$ and $Y$ respectively. 

We can express composition of functors in terms of their kernels. If $X,Y,Z $ are varieties and $\Phi_\sP : D(\O_X\mod) \rightarrow D(\O_Y\mod),  \Phi_\sQ : D(\O_Y\mod) \rightarrow D(\O_Z\mod) $ then $ \Phi_\sQ \circ \Phi_\sP $ is induced by the kernel
\begin{equation*}
\sQ * \sP := {\pi_{13}}_*(\pi^*_{12}(\sP) \otimes \pi^*_{23}(\sQ))
\end{equation*}
where $*$ is called the convolution product. We now describe the analogous formalism for $\D_h$-modules.

\subsection{Kernels for $\D_{X,h}$-modules}\label{sec:kerD}
Consider two smooth and proper varieties $X$ and $Y$ and an object $\sP \in D(\D_{X \times Y,h}\mod)$. Such a kernel induces a functor $\Phi_{\sP}: D(\D_{X,h}\mod) \rightarrow D(\D_{Y,h}\mod)$ via
$$\Phi_\sP(M) = \int_{p_2} p_1^\d(M) \dotimes_{\O_{X \times Y,h}} \sP $$
where $M \in D(\D_{X,h}\mod)$. As above, if $\sP \in D(\D_{X \times Y,h}\mod)$ and $\sQ \in D(\D_{Y \times Z,h}\mod)$ then we can define
$$\sQ * \sP := \int_{p_{13}} p^\d_{12}(\sP) \dotimes_{\O_{X \times Y \times Z,h}} p^\d_{23}(\sQ).$$
The following result follows from a formal argument which is exactly the same as in the case of quasi-coherent $\O$-modules.

\begin{Lemma}
If $\sP \in D(\D_{X \times Y,h}\mod)$ and $\sQ \in D(\D_{Y \times Z,h}\mod)$ then $\Phi_{\sQ} \circ \Phi_{\sP} \cong \Phi_{\sQ * \sP}$.
\end{Lemma}

\subsection{The associated graded of kernels} \label{se:AssociatedGradedKernels}

For a kernel $\sP \in D(\D_{X \times Y,h}\mod)$ we define the ``directed'' associated graded by
\begin{equation}\label{eq:tgr}
\tgr(\sP) := (1 \times \iota)^* \gr(\sP) \otimes \omega_Y [-d_X] \in D(\O_{T^* X \times T^* Y}\mod)
\end{equation}
where $\iota: T^* Y \rightarrow T^* Y$ is the involution acting by multiplication by $(-1)$ on the fibres and $\omega_{Y}$ is the canonical bundle pulled back from $Y$ via the natural projection $T^* X \times T^* Y \rightarrow Y$. We say ``directed'' because $\tgr(\sP)$ depends on whether we think of $\sP$ as an object on $X \times Y$ or $Y \times X$.

\begin{Proposition}\label{prop:1} Let $X,Y,Z$ be smooth proper varieties with $\sP \in D(\D_{X \times Y,h}\mod)$ and $\sQ \in D(\D_{Y \times Z,h}\mod)$. Then $\tgr(\sQ * \sP) \cong \tgr(\sQ) * \tgr(\sP)$.
\end{Proposition}
\begin{proof}
We will write $d_X,d_Y,d_Z$ for the dimensions of $X,Y,Z$. We begin by first calculating $\gr(p_{12}^\d (\sP) \dotimes_{\O_{X \times Y \times Z,h}} p_{23}^\d (\sQ))$. By Proposition \ref{prop:assgrtensor} we have that
$$\gr(p_{12}^\d(\sP) \otimes_{\O_{X \times Y \times Z,h}} p_{23}^\d(\sQ)) \cong \rho_* \pi_1^* (\gr(p_{12}^\d \sP)) \otimes \pi_2^*( \gr(p_{23}^\d \sQ))$$
where
$$\pi_1, \pi_2, \rho: T^*(X \times Y \times Z) \times_{X \times Y \times Z} T^*(X \times Y \times Z) \rightarrow T^*(X \times Y \times Z)$$
are the natural maps as in section \ref{sec:tensor}.

Now by Proposition \ref{prop:pull} we have $\gr(p_{12}^\d \sP) \cong i_{12*} \pi_{12}^* (\gr \sP)[d_Z]$ where $i_{12}$ and $\pi_{12}$ are the natural inclusion and projection maps $T^*(Y \times X \times Z) \xleftarrow{i_{12}} T^*(X \times Y) \times Z \xrightarrow{\tpi_{12}} T^*(X \times Y)$. Similarly $\gr(p_{23}^\d \sP) \cong i_{23*} \tpi_{23}^* (\gr \sP)[d_X]$ so we get
$$\gr(p_{12}^\d(\sP) \dotimes_{\O_{X \times Y \times Z,h}} p_{23}^\d(\sQ)) \cong \rho_*(\pi_1^* i_{12*} \tpi_{12}^* (\gr \sP) \otimes \pi_2^* i_{23*} \tpi_{23}^* (\gr \sQ)) [-d_Y].$$

Next we take a fibre product to get
\begin{equation*}
\xymatrix{
[T^* (X \times Y) \times Y] \times_{X \times Y \times Z} T^*(X \times Y \times Z) \ar[r]^{\ti_{12}} \ar[d]^{\tpi_1} & T^*(X \times Y \times Z) \times_{X \times Y \times Z} T^*(X \times Y \times Z) \ar[d]^{\pi_1} \\
T^* (X \times Y) \times Z \ar[r]^{i_{12}} & T^*(X \times Y) \times T^* Z
}
\end{equation*}
Since $\pi_1$ is flat we get that $\pi_1^* i_{12*} (\cdot) \cong \ti_{12*} \tpi_1^* (\cdot)$. Similarly we get $\pi_2^* i_{23*} (\cdot) \cong \ti_{23*} \tpi_2^* (\cdot)$. Putting this together gives
\begin{align*}
\gr(p_{12}^\d(\sP) \dotimes_{\O_{X \times Y \times Z,h}} p_{23}^\d(\sQ))
\cong& \rho_* (\ti_{12*} \tpi_1^* \tpi_{12}^* (\gr \sP) \otimes \ti_{23*} \tpi_2^* \tpi_{23}^* (\gr \sQ)) [-d_Y] \\
\cong& \rho_* \ti_{12*} (\tpi_1^* \tpi_{12}^* (\gr \sP) \otimes \ti_{12}^* \ti_{23*} \tpi_2^* \tpi_{23}^* (\gr \sQ)) [-d_Y].
\end{align*}

To compute $\ti_{12}^* \ti_{23*}(\cdot)$ we use the following fibre product
\begin{equation*}
\xymatrix{
[T^*(X \times Y) \times Z] \times_{X \times Y \times Z} [X \times T^*(Y \times Z)] \ar[d]^{\hi_{23}} \ar[r]^{\hi_{12}} & T^*(X \times Y \times Z) \times_{X \times Y \times Z} [X \times T^*(Y \times Z)] \ar[d]^{\ti_{23}} \\
[T^* (X \times Y) \times Z] \times_{X \times Y \times Z} T^*(X \times Y \times Z) \ar[r]^{\ti_{12}} & T^*(X \times Y \times Z) \times_{X \times Y \times Z} T^*(X \times Y \times Z).
}
\end{equation*}
Since the images of $\ti_{12}$ and $\ti_{23}$ meet transversely we have $\ti_{12}^* \ti_{23*} (\cdot) \cong \hi_{23*} \hi_{12}^* (\cdot)$ and so we get
\begin{align*}
\gr(p_{12}^\d(\sP) \dotimes_{\O_{X \times Y \times Z,h}} p_{23}^\d(\sQ))
&\cong \rho_* \ti_{12*} (\tp_1^* \tpi_{12}^* (\gr \sP) \otimes \hi_{23*} \hi_{12}^* \tpi_2^* \tpi_{23}^* (\gr \sQ)) [-d_Y] \\
&\cong \rho_* \ti_{12*} \hi_{23*} (\hi_{23}^* \tpi_1^* \tpi_{12}^* (\gr \sP) \otimes \hi_{12}^* \tpi_2^* \tpi_{23}^* (\gr \sQ)) [-d_Y].
\end{align*}
Now $\tpi_{12} \circ \tp_1 \circ \hi_{12}$ is precisely the projection
$$q_{12}: [T^*(X \times Y) \times Z] \times_{X \times Y \times Z} [X \times T^*(Y \times Z)] \rightarrow T^*(X \times Y)$$
onto the first factor while $\tpi_{23} \circ \tp_2 \circ \hi_{23}$ is the projection
$$q_{23}: [T^*(X \times Y) \times Z] \times_{X \times Y \times Z} [X \times T^*(Y \times Z)] \rightarrow T^*(Y \times Z).$$
Also
$$q_{13} := \rho \circ \ti_{12} \circ \hi_{23}: T^*(X \times Z) \times [T^* Y \times_{Y} T^* Y] \rightarrow T^*(X \times Z) \times T^* Y$$
is just the map $1 \times \rho_Y$ where $\rho_Y: T^* Y \times_{Y} T^* Y \rightarrow T^* Y$ is the map from section \ref{sec:tensor}. So we get
$$\gr(p_{12}^\d(\sP) \dotimes_{\O_{X \times Y \times Z,h}} p_{23}^\d(\sQ)) \cong q_{13*}(q_{12}^*(\gr \sP) \otimes q_{23}^*(\gr \sQ)) [-d_Y].$$

Finally, we have
\begin{align*}
\tgr(\sQ * \sP)
\cong& (1 \times \iota)^* \gr(p_{13*}(p_{12}^\d(\sP) \dotimes_{\O_{X \times Y \times Z,h}} p_{23}^\d(\sQ))) \otimes \omega_Z [-d_X] \\
\cong& (1 \times \iota)^* \tpi_{13*} (i_{13}^* (\gr(p_{12}^\d(\sP) \dotimes_{\O_{X \times Y \times Z,h}} p_{23}^\d(\sQ))) \otimes \omega_{Y}) \otimes \omega_{Z} [-d_X]
\end{align*}
where $i_{13}$ and $\tpi_{13}$ are the natural maps $T^*(X \times Y \times Z) \xleftarrow{i_{13}} T^*(X \times Z) \times Y \xrightarrow{\tpi_{13}} T^*(X \times Z)$. Using the above result this gives
$$\tgr(\sQ * \sP) \cong (1 \times \iota)^* \tpi_{13*} (i_{13}^* q_{13*}(q_{12}^*(\gr \sP) \otimes q_{23}^*(\gr \sQ)) \otimes \omega_{Y}) \otimes \omega_{Z} [-d_X-d_Y].$$
To compute $i_{13}^* q_{13*} (\cdot)$ we use the following fibre diagram
\begin{equation*}
\xymatrix{
T^* (X \times Z) \times T^* Y \ar[rr]^{\ti_{13}} \ar[d]^{\tq_{13}} & & T^*(X \times Z) \times [T^* Y \times_{Y} T^* Y]  \ar[d]^{q_{13}} \\
T^*(X \times Z) \times Y \ar[rr]^{i_{13}} & & T^*(X \times Z) \times T^* Y
}
\end{equation*}
where $\ti_{13}$ is induced by the map $T^* Y \rightarrow T^* Y \times_{Y} T^* Y$ which takes $(p,v) \mapsto (p,-v,v)$. Since $q_{13}$ is flat we get $i_{13}^* q_{13*} (\cdot) \cong \tq_{13*} \ti_{13}^* (\cdot)$ and hence
$$\tgr(\sQ * \sP) \cong (1 \times \iota)^* \tpi_{13*} (\tq_{13*} (\ti_{13}^* q_{12}^* (\gr \sP) \otimes \ti_{13}^* q_{23}^* (\gr \sQ)) \otimes \omega_{Y} ) \otimes \omega_{Z} [-d_X-d_Y].$$
Now $q_{12} \ti_{13}: T^*(X \times Y \times Z) \rightarrow T^*(X \times Y)$ is equal to $(1 \times \iota) \pi_{12}$ and $q_{23} \ti_{13} = \pi_{23}$ where $\pi_{ij}$ are the natural projections. Also, $\tpi_{13
} \tq_{13} = \pi_{13}$ and so we get
\begin{eqnarray*}
\tgr(\sQ * \sP)
&\cong& (1 \times \iota)^* \pi_{13*}(\pi_{12}^* (1 \times \iota)^* (\gr \sP) \otimes \pi_{23}^* (\gr \sQ) \otimes \omega_{Y}) \otimes \omega_{Z} [-d_X-d_Y] \\
&\cong& \pi_{13*}((1 \times 1 \times \iota)^* \pi_{12}^*((1 \times \iota)^* (\gr \sP) \otimes \omega_{Y} [-d_X]) \otimes (1 \times 1 \times \iota)^* \pi_{23}^* (\gr \sQ) \otimes \omega_{Z} [-d_Y]) \\
&\cong& \pi_{13*}(\pi_{12}^* ((1 \times \iota)^* (\gr \sP) \otimes \omega_{Y} [-d_X]) \otimes \pi_{23}^*((1 \times \iota)^* (\gr \sQ) \otimes \omega_{Z} [-d_Y])) \\
&\cong& \pi_{13*}(\pi_{12}^*(\tgr \sP) \otimes \pi_{23}^*(\tgr \sQ)) \\
&\cong& \tgr(\sQ) * \tgr(\sP).
\end{eqnarray*}
\end{proof}

\section{Mixed Hodge modules}\label{sec:mhm}

We will need some results from Saito's theory of pure and mixed Hodge modules \cite{Sa1,Sa2} which we now recall.  For any smooth variety $X$, Saito constructed abelian categories of polarizable pure Hodge modules of weight $ i $, $\HM(X,i) $,  and polarizable mixed Hodge modules, $\MHM(X)$. We have an exact forgetful functor to filtered $\D_X$-modules and subsequently a functor $\fG: \MHM(X) \rightarrow \D_{X,h}\mod$. This functor is exact because for any morphism in $\MHM$, the associated morphism of filtered $\D$-modules is strict with respect to the filtration. We denote by $\{1\} : \MHM(X) \rightarrow \MHM(X)$ the functor which shifts the filtration. This implies that $\fG(M\{1\}) = \fG(M)\{1\}$ for any $M \in \MHM(X)$.

An object $M \in \MHM(X)$ is equipped with an increasing filtration, denoted $W_i(M)$, called the {\it weight filtration}. The subquotients $ \gr_i^W(M) = W_i(M)/W_{i-1}(M) $ lie in the category $\HM(X,i) $.

Saito shows in \cite{Sa2} that for any morphism $f: X \rightarrow Y$ there exist associated functors $f_*: D(\MHM(X)) \rightarrow D(\MHM(Y))$ and $f^!: D(\MHM(Y)) \rightarrow D(\MHM(X))$.

The functor $ f_* $ is compatible with the push-forward of $ \D_{X,h}$-modules for proper $ f $ and is compatible with the push-forward of $ \D_X$-modules for any $ f $ (the latter result follows immediately from the definitions).
\begin{Theorem} \cite[Thm. 4.3]{Sa2} \label{thm:pushHodge}
If $f: X \rightarrow Y$ is proper, then
$$\fG \circ f_* = \int_f \circ \; \fG: D(\MHM(X)) \rightarrow D(\D_{Y,h}\mod).$$
\end{Theorem}

\begin{Proposition} \label{prop:pushHodgeD}
If $ f : X \rightarrow Y $ is any morphism, then for $ M \in D(\MHM(X)) $,
$$\fG(f_*(M)) \otimes_{\C[h]} \C_1 = \int_f \bigl( \fG(M) \otimes_{\C[h]} \C_1 \bigr)$$
\end{Proposition}

\begin{Remark}
The pullback $ f^!$ of mixed Hodge modules is similarly compatible with the pullback $ f^\dagger $ of $ \D$-modules (though we will not directly use this fact in the paper).
\end{Remark}

\subsection{Structure theorem and decomposition theorem}

An object $ M \in \HM(X,w)$ has strict support $ Z \subseteq X $ if it has no non-zero sub or quotient object supported on proper closed subvarieties of $ Z $.  From the definition of $ \HM(X,w)$, we see that every object $ M \in \HM(X,w) $ admits a decomposition $ M = \oplus_{Z \subseteq X} M_Z $, where $ M_Z $ has strict support on $Z$ (see \cite[section 12]{Sc}).

The following structure theorem appears in \cite{Sa2} (see \cite[Theorem 15.1]{Sc}).

\begin{Theorem} \label{th:structure}
Let $ \mathcal V $ be a variation of Hodge structure of weight $ w $ on a smooth open subset $ U \subset Z $.  Then $ \mathcal V $ extends uniquely to a Hodge module on $ X $ of weight $ w + d_Z$ with strict support $ Z$.  Moreover, every Hodge module with strict support $ Z $ is obtained this way.
\end{Theorem}

In particular, if we start with the constant rank 1 variation of Hodge structure on an open subset of $ Z$ of weight $ -d_Z$, we obtain a Hodge module on $ Z $ which we denote by $ \IC_{Z,m} $. We define $ \IC_{Z,h} = \fG(\IC_{Z,m}) $.  In the case, when $ Z $ is smooth, we will write $ \delta_{Z,m} $ for $ \IC_{Z,m} $ --- note that in this case $ \IC_{Z,h} = \delta_{Z,h} $, by Theorem \ref{thm:pushHodge}, so the notation is consistent.

Saito also proved the following decomposition theorem for projective morphisms.

\begin{Theorem} \cite[Section 5]{Sa1} \label{th:decompHodge}
If $ M \in \HM(X,w) $ and $ f : X \rightarrow Y $ is projective, then each $ H^i(f_*(M)) $ lies in $\HM(Y,w + i)$.  Moreover, $ f_*(M) = \oplus_i H^i(f_*(M))[-i] $.
\end{Theorem}

\begin{Corollary} \label{cor:small}
If $ f : X \rightarrow Y $ is a small resolution then $ f_* (\O_{X,m}) \cong \IC_{Y, m} $.
\end{Corollary}
\begin{proof}
It suffices to show that the left hand side is indecomposable (because of Theorems \ref{th:structure} and \ref{th:decompHodge}). Suppose it is decomposable. Then $\fG(f_* \O_{X,m}) \otimes_{\C[h]} \C_1$ is decomposable (here we use that $\fG(\cdot)$ is $\C[h]$-torsion free). But
$\fG(f_* \O_{X,m}) \otimes_{\C[h]} \C_1 \cong \int_f \O_X$ by Proposition \ref{prop:pushHodgeD}. The right hand side is just $\IC_Y$ which is indecomposable (contradiction).
\end{proof}

\subsection{Kashiwara's theorem and base change}
Finally, we recall the base change theorem in the category $\MHM$. If $i: Y \rightarrow X$ is an inclusion of smooth varieties then \cite[Formula 4.24]{Sa2} states that $i_*: \MHM(Y) \rightarrow \MHM_Y(X)$ is an equivalence of categories, where the right hand side denotes the full subcategory of $\MHM(X)$ consisting of objectes which are set-theoretically supported on $Y$.

\begin{Theorem} \label{th:KashforHodge}
If $i: Y \rightarrow X$ is an inclusion of smooth varieties then the functor $i^* = i^!$, when restricted to $\MHM_Y(X)$, is the inverse to $i_*$.
\end{Theorem}
\begin{proof}
Let $j: U \subset X$ denote the complement to $Y$. According to \cite[Section 4.4]{Sa2} we have an exact triangle for any $M$ in $D(\MHM(X))$
$$j_!j^{!}(M) \to M \to i_*i^*(M)$$
which implies that the natural map $M \to i_*i^*(M)$ is an isomorphism for any $M \in \MHM_Y(X)$.  Now, if we take any $N$ in $\MHM(Y)$, the above implies that we have an isomorphism
$$i_*i^*i_*(N) \to i_*(N)$$
which means, since $i_*$ is fully faithful, that the natural map $i^*i_*(N) \to N$ is an isomorphism. This proves the assertion for $i^*$. The proof for $i^!$ is the same but with the above exact triangle replaced by its dual
$$i_*i^!(M) \to M \to j_*j^{!}(M).$$
\end{proof}

\begin{Remark}
The theorem above holds for $\D$-modules (this is Kashiwara's theorem) but does not hold for $\D_h$-modules. One reason is that (in general) the analogue of Theorem \ref{thm:pushHodge} does not hold for pullbacks (meaning that pullback does not commute with $\fG$).
\end{Remark}

\begin{Theorem} \label{thm:basechangeforHodge}
In the fiber product diagram below we have $g^{!} f_* \cong f'_* (g')^!$.
$$
\xymatrix{
Y \times_X Z \ar[rrr]^{g'} \ar[d]_{f'} & & & Y \ar[d]^f \\
Z \ar[rrr]^g & & & X }$$
\end{Theorem}
\begin{proof}
The proof of this is the same as for $\D$-modules (see \cite[Sec. 1.7]{HTT}) and relies only on Kashiwara's theorem (Theorem \ref{th:KashforHodge}) and the adjunction triangles for pushforwards and pullbacks.
\end{proof}

\section{The $\sl_2$ action on $\D_h$-modules}\label{sec:actionDmod}

\subsection{The action}
We will now define a categorical $\sl_2$ action. Fix an integer $N$ and define a 2-category $ \K_\D $ with objects $ \l = -N, -N+2, \dots, N-2, N $ and 1-morphism categories $\K_\D(\l, \l') := D(\D_{\G(k,N) \times \G(k',N),h}) $ where $ k,k' $ are related to $ \l, \l' $ by $ \l = N -2k, \l' = N - 2k' $. We will also write $d_k := k(N-k) = \dim (\G(k,N))$.

We have the following correspondence
\begin{equation*}
\G(k,N) = \{0 \subset V \subset \C^N \} \xleftarrow{q_1} C(\l) := \{0 \subset V' \subset V \subset \C^N \} \xrightarrow{q_2} \{0 \subset V' \subset \C^N \} = \G(k-1,N)
\end{equation*}
where $q_1$ and $q_2$ forget $V'$ and $V$ respectively and $\l=N-2k+1 = d_k-d_{k-1}$. These can be used to define kernels
\begin{align*}
& \sE(\l) := \delta_{C(\l),h} \la d_k \ra \in D(\D_{\G(k,N) \times \G(k-1,N),h}\mod) \\
& \sF(\l) := \delta_{C(\l),h} \la d_{k-1} \ra \in D(\D_{\G(k-1,N) \times \G(k,N),h}\mod).
\end{align*}

\begin{Lemma}\label{lem:adjoint}
\begin{enumerate}
\item We have isomorphisms of functors 
$$ \Phi_{\sE(\l)} \cong \int_{q_2} q_1^*(\cdot) \quad \text{ and } \quad \Phi_{\sF(\l)} \cong \int_{q_1} q_2^*(\cdot).$$

\item The left and right adjoints of $ \Phi_{\sE(\l)} $ are given by the kernels
$\sF(\l) \la -\l \ra $ and $\sF(\l) \la \l \ra.$
\end{enumerate}
\end{Lemma}

\begin{proof}
\begin{enumerate}
\item Let $ M \in D(\D_{\G(k,N), h}\mod) $.  Consider the closed embedding $$ q : C(\l) \hookrightarrow \G(k,N) \times \G(k-1,N) $$ Note that $ q_1 = p_1 \circ q $ and $ q_2 = p_2 \circ q $ and recall that by definition $ \delta_{C(\l),h} = \int_q \O_{C(\l),h} \{ \dim C(\l) \} $.  Applying the projection formula (Proposition \ref{prop:projection}) to $ q $, we have
\begin{align*}
\Phi_{\sE(\l)}(M) &= \int_{p_2} (p_1^\dagger(M) \dotimes_{\O_{\G(k,N) \times \G(k-1,N),h}} \int_q \O_{C(\l), h} \{\dim C(\l)\} \la d_k \ra) \\
&\cong \int_{p_2} \int_q  (q^\dagger p_1^\dagger(M) \dotimes_{\O_{C(\l),h}} \O_{C(\l), h} \{\dim C(\l)\} \la d_k \ra) \\
&\cong \int_{p_2} \int_q  (q^\dagger p_1^\dagger(M) [-\dim C(\l)]  \{\dim C(\l)\} \la d_k \ra) \\
&\cong \int_{q_2}  q_1^\dagger(M) \la d_k  - \dim C(\l)  \ra = \int_{q_2}  q_1^*(M).
\end{align*}
The second isomorphism is similar.

\item
We begin with the left adjoint.  Let $M \in D(\D_{\G(k,N),h}\mod), N \in D(\D_{\G(k-1,N),h}\mod) $. 

Since $ q_1 $ is proper and $ q_2$ is smooth, by  Propositions \ref{pr:adj1} and \ref{pr:adj2}, we deduce that
\begin{align*}
 \Hom_{\D_{\G(k,N),h}}(\int_{q_1} q_2^\dagger (N), M) &\cong \Hom_{\D_{C(\l),h}}(q_2^\dagger(N), q_1^\dagger(M)) \\
&\cong \Hom_{\D_{\G(k-1, N),h}}(N , \int_{q_2} q_1^\dagger(M)\la 2(d_{k-1} - \dim C(\l)) \ra)
\end{align*}
Since $ \l = d_k -d_{k-1} $, this is equivalent to
$$
\Hom_{\D_{\G(k,N),h}}(\int_{q_1} q_2^\dagger(N)\la d_{k-1} - \dim C(\l) - \l \ra, M) \cong \Hom_{\D_{\G(k-1,N),h}}(N, \int_{q_2} q_1^\dagger(M)\la d_k - \dim C(\l) \ra)
$$
which implies the result by part (i). 

The right adjoint is similar.
\end{enumerate}
\end{proof}

\begin{Remark}
In general there is no nice expression in terms of kernels for the left and right adjoints of a functor. The situation in Lemma \ref{lem:adjoint} is special since the functor corresponding to the kernel is a pullback followed by a pushforward via smooth maps.
\end{Remark}

\subsection{Proof of the commutation relation}

To prove the main $\sl_2$ commutation relation we first need the following fact about small resolutions of $ GL_N $ orbit closures in products of Grassmannians. This is closely related to a result of Zelevinsky \cite{Ze}. Fix two integers $k, l $ with $ k \le l $. The $ GL_N $ orbits in $\G(k,N) \times \G(l,N) $ are the strata $Z_s := \{ (V, V') : \dim V \cap V' = s \}$, where $ s = \max(0, k+l-N), \dots, k $. We have $\oZ_s = Z_s \cup \dots \cup Z_k $. One can define a resolution $\pi: P_s \rightarrow \oZ_s$ where
$$P_s := \{(V, V', V'') \in \G(k,N) \times \G(l,N) \times \G(s, N) : V'' \subset V \cap V' \}.$$
and $\pi$ forgets $V''$.

\begin{Proposition} \label{pr:smallres}
If $k+l\le N$ then the map $\pi: P_s \rightarrow \oZ_s $ is a small resolution.
\end{Proposition}

\begin{proof}
Since $P_s$ is an iterated bundle of Grassmannians it must be smooth. It is clear that $\pi$ is one-to-one over $Z_s$ which means that $\pi$ is a resolution. It remains to show that $\pi$ is small.

The relevant strata are the $ Z_t \subset \overline{Z_s} $ for $ t > s $.  An elementary computation (based on $ \dim (Z_s) = \dim (P_s) $ and the description of $ P_s $ as an iterated Grassmannian bundle) shows that
$$ \dim (Z_s) - \dim (Z_t) = (t-s)(s+t+N-k-l).$$
On the other hand, for a point $ (V, V') \in Z_t$, we see that $ \pi^{-1}(V,V') $ is the Grassmannian of $s$-dimensional subspaces of $ V \cap V'$.  Since $ \dim V \cap V' = t$, we see that $ \dim \pi^{-1}(V,V') = s(t-s)$. Since $k+l \le N $ and $ t < s $, we see that $ 2s < s+t +N -k -l$ and conclude that
$$2 \dim \pi^{-1}(V,V') < \dim (Z_s) - \dim (Z_t)$$
which proves that $\pi$ is small.
\end{proof}

\begin{Proposition}\label{prop:sl2Grass}
We have the following relations
\begin{enumerate}
\item $\sF(\l+1) * \sE(\l+1) \cong \sE(\l-1) * \sF(\l-1) \bigoplus_{[-\l]}  \delta_{\Delta,h} \la d_k \ra$ if $\l \le 0$,
\item $\sE(\l-1) * \sF(\l-1) \cong \sF(\l+1) * \sE(\l+1) \bigoplus_{[\l]} \delta_{\Delta,h} \la d_k \ra $ if $\l \ge 0$
\end{enumerate}
inside $D(\D_{\G(k,N) \times \G(k,N),h}\mod)$, where $\l = N-2k$.
\end{Proposition}
\begin{proof}
We will prove that case $\l \ge 0$ (the case $\l \le 0$ is the same). First, we have
\begin{align*}
\sF(\l+1) * \sE(\l+1)
&\cong \int_{p_{13}} p_{12}^\d \delta_{C(\l+1),h} \dotimes p_{23}^\d \delta_{C(\l+1),h} \la d_k+d_{k-1} \ra \\
&\cong \int_{p_{13}} \delta_{p_{12}^{-1}(C(\l+1)),h} \la d_k \ra \dotimes \delta_{p_{23}^{-1}(C(\l+1)),h} \la d_k \ra \la d_k+d_{k-1} \ra \\
&\cong \int_{p_{13}} \delta_{P_{k-1},h} \la d_k \ra
\end{align*}
where $P_{k-1} = \{0 \subset V' \subset V,V'' \subset \C^N\}$ with $\dim(V')=k-1, \dim(V/V')=\dim(V''/V')=1$. Note that the third isomorphism follows from Corollary \ref{cor:Deltas} and a dimension count (the tensor product in the second line is happening in the triple product which has dimension $2d_k+d_{k-1}$). 

Since $p_{13}$ forgets $V'$ the image $p_{13}(P_{k-1})$ is equal to $\overline{Z_{k-1}} = \{0 \subset V,V'' \subset \C^N: \dim(V \cap V'') \ge k-1\}$. By Proposition \ref{pr:smallres} the map $\pi: P_{k-1} \rightarrow \overline{Z_{k-1}} $ is small. Subsequently, by Corollary \ref{cor:small} we get $\int_{p_{13}} \delta_{P_{k-1},h} = \IC_{\overline{Z_{k-1}},h}$ and thus $\sF(\l+1) * \sE(\l+1) \cong \IC_{\overline{Z_{k-1}},h} \la d_k \ra$.

On the other hand, a similar argument shows
$$\sE(\l-1)*\sF(\l-1) \cong \int_{p_{13}} \delta_{P',h} \la d_k \ra$$
where $P' := \{0 \subset V,V'' \subset V' \subset \C^N \}$ with $\dim(V')=k+1, \dim(V'/V)=\dim(V'/V'')=1$. The pushforward $\int_{p_{13}} \delta_{P',h}$ (which is now more difficult to calculate because the map is no longer small) is computed in Lemma \ref{lem:pushforward}. The result follows.
\end{proof}

\begin{Lemma}\label{lem:pushforward}
Let $P' = \{0 \subset V,V'' \subset V' \subset \C^N \} \subset \G(k,N) \times \G(k+1,N) \times \G(k,N)$ and consider the projection $p_{13}$ which forgets $V'$ as in the proof of Proposition \ref{prop:sl2Grass}. Then
$$\int_{p_{13}} \delta_{P',h} \cong \IC_{\overline{Z_{k-1}},h} \bigoplus_{[N-2k]} \int_{\Delta} \delta_{\G(k,N),h} .$$
\end{Lemma}
\begin{proof}
We will show the corresponding result
\begin{equation}\label{eq:toshow}
p_{13*} (\delta_{P',m}) \cong \IC_{\overline{Z_{k-1}},m} \bigoplus_{[N-2k]} \Delta_* (\delta_{\G(k,N),m})
\end{equation}
at the level of MHM. The result then follows by applying $\fG$ and using that $\fG$ commutes with proper pushforward.

The decomposition theorem for Hodge modules tells us that
\begin{equation}\label{eq:4}
p_{13*} (\delta_{P',m}) \cong \oplus_i H^i(p_{13*} (\delta_{P',m}))[-i]
\end{equation}
and the structure theorem tells us that
\begin{equation}
H^i(p_{13*} (\delta_{P',m})) = N_i \oplus D_i
\end{equation}
where $N_i $ has strict support $ \overline{Z_{k-1}} $ and $ D_i $ has strict support the diagonal $ \G(k,N)$ (these are the only possible strict supports because the whole situation is $ GL_N $ equivariant and the image of $ p_{13} $ is $\overline{Z_{k-1}} $).

Since $p_{13*} $ is an isomorphism over $ Z_{k-1}$, the structure theorem for Hodge modules shows us that $ N_i = 0 $ for $ i \ne 0 $ and that $ N_0 = \IC_{\overline{Z_{k-1}},m}$. Thus we conclude that
$$ p_{13*} (\delta_{P',m}) = \IC_{\overline{Z_{k-1}},m} \oplus_i D_i[-i] $$
where each $ D_i $ has strict support along the diagonal.  We will now consider the base change over the diagonal. Using the following commutative diagram
$$\xymatrix{
Q := \{0 \subset V=V'' \subset V' \subset \C^N\} \ar[r]^-{\tilde{i}} \ar[d]^{\widehat{\Delta}} & \G(k,N) \times \G(k+1,N) \ar[d]^{\widetilde{\Delta}} \ar[r]^-{\tp_{13}} & \G(k,N) \ar[d]^{\Delta} \\
P' \ar[r]^-{i} & \G(k,N) \times \G(k+1,N) \times \G(k,N) \ar[r]^-{p_{13}} & \G(k,N) \times \G(k,N)}$$
together with the base change formula we have
\begin{align*}
\Delta^! p_{13*} i_{*} \delta_{P',m}
&\cong \tp_{13*} \widetilde{\Delta}^! i_{*} \delta_{P',m} \cong (\tp_{13})_{*} {\tilde{i}_*} \widehat{\Delta}^! \delta_{P',m} \\
&\cong \tp_{13*} {\tilde{i}_*} \delta_{Q,m} \langle -k \rangle \\
&\cong \bigoplus_{[N-k]} \delta_{\G(k,N),m} \la -k \ra .
\end{align*}
where the third isomorphism follows from the Hodge module analogue of Corollary \ref{cor:pullbackdelta} and the last isomorphism uses the fact that $\tp_{13} \circ \tilde{i}: Q \rightarrow \G(k,N)$ is a $\P^{N-k-1}$ fibration and gives a constant variation of Hodge structure with fibre $ H^* (\P^{N-k-1}) $.  Note that as a Hodge structure,
$$ H^* (\P^r) = \C[-r]\{r\} \oplus \C[2-r]\{r-2\} \cdots \oplus \C[r]\{-r\} = \bigoplus_{[r+1]} \C  $$

A similar analysis shows that $\Delta^! p_{13*} i_* \delta_{P_{k-1},m} \cong \bigoplus_{[k]} \delta_{\G(k,N),m} \la -N+k \ra $. Since $p_{13*} \delta_{P_{k-1},m} = \IC_{\overline{Z_{k-1},m}}$ we also have
$$\Delta^! \IC_{\overline{Z_{k-1},m}} \cong \bigoplus_{[k]} \delta_{\G(k,N),m} \la -N+k \ra.$$

Thus, we conclude that
$$
\bigoplus_{[N-k]} \delta_{\G(k,N),m} \la -k \ra = \bigoplus_{[k]} \delta_{\G(k,N),m} \la -N+k \ra \oplus_i D_i[-i]
$$
where each $ D_i $ has strict support along the diagonal.  Thus we conclude that
$$
\oplus_i D_i[-i] = \bigoplus_{[N-2k]} \delta_{\G(k,N),m}
$$
The relation in (\ref{eq:toshow}) now follows.
\end{proof}

\subsection{Divided powers}\label{sec:divided}

It is interesting in this case to identify explicitly the divided powers $\sE^{(r)}$ and $\sF^{(r)}$ for $r \in \N$. Consider the following correspondences
\begin{equation*}
\G(k,N) = \{0 \subset V \subset \C^N \} \xleftarrow{p_1} C^r(\l) := \{0 \subset V' \subset V \subset \C^N \} \xrightarrow{p_2} \{0 \subset V' \subset \C^N \} = \G(k-r,N)
\end{equation*}
where $\l=N-2k+r$ (these generalize $C(\l)$ from above). It turns out that
\begin{align*}
& \sE^{(r)}(\l) := \delta_{C^r(\l),h} \la d_k \ra \in D(\D_{\G(k,N) \times \G(k-r,N),h}\mod) \\
& \sF^{(r)}(\l) := \delta_{C^r(\l),h} \la d_{k-r} \ra \in D(\D_{\G(k-r,N) \times \G(k,N),h}\mod).
\end{align*}
By a similar argument as in the last section we have
$$\sE^{(r_2)} * \sE^{(r_1)} \cong \int_{p_{13}} \delta_{U,h} \la d_k \ra$$
where $U := \{0 \subset V'' \subset V' \subset V \subset \C^N\}$ with $\dim(V)=k+r_1, \dim(V/V')=r_1, \dim(V'/V'')=r_2$.
Since $p_{13}$ forgets $V'$ we have $p_{13}(U) = C^{r_1+r_2}$ where the projection $U \rightarrow p_{13}(U)$ is a $\G(r_1,r_1+r_2)$-bundle. Thus
$$\int_{p_{13}} \delta_{U,h} \cong \bigoplus_{\qbins{r_1+r_2}{r_1}} \delta_{C^{r_1+r_2},h}$$
since, as a Hodge structure,
$$H^* (\G(r_1,r_1+r_2)) = \bigoplus_{\qbins{r_1+r_2}{r_1}} \C.$$
Thus we end up with
$$\sE^{(r_2)} * \sE^{(r_1)} \cong \bigoplus_{\qbins{r_1+r_2}{r_1}} \delta_{C^{r_1+r_2},h} \la d_k \ra \cong \bigoplus_{\qbins{r_1+r_2}{r_1}} \sE^{(r_1+r_2)}.$$

\subsection{Action of the nil affine Hecke algebra}

This action can be obtained by using the main result of \cite{C2}. However, there is a more direct proof available using the fact that at the level of constructible sheaves this action is well known and relatively straight-forward to check. Via the Riemann-Hilbert correspondence one also gets an action in the context of $\D$-modules. We now explain how to lift this to an action on $\D_h$-modules.

The key to doing this is the fact that, for a smooth, proper variety $X$, we have
$$\Ext_{\D_{X,h}}^k(\O_{X,h}, \O_{X,h}) \cong \bigoplus_{i+j=k} H^{i,j}(X) \otimes_{\C} \C[h] \{2i\}.$$
As in the case of standard $\D$-modules (when $h=1$) one can see this using the Spencer resolution from Corollary \ref{co:SpPoint}
$$0 \rightarrow \D_{X,h} \otimes_{\O_{X,h}} \bigwedge^{d_X} \Theta_{X,h} \{-2 d_X\} \rightarrow \dots \rightarrow \D_{X,h} \otimes_{\O_{X,h}} \Theta_{X,h} \{-2\} \rightarrow \D_{X,h} \rightarrow \O_{X,h} \rightarrow 0$$
In particular, if $H^{0,2}(X)=0=H^{2,0}(X)$ then the space of maps
$$\O_{X,h} \to \O_{X,h}[2]\{-2\}$$
can be identified with $H^{1,1}(X)$.

In our case, this gives us that the space of maps
$$\delta_{C(\l),h} \to \delta_{C(\l),h} [2]\{-2\}$$
can be identified with $H^{1,1}(C(\l))$. In particular, this allows us to define
$$X(\l): \delta_{C(\l),h} \to \delta_{C(\l),h} [2]\{-2\}$$
by lifting the $X(\l): \delta_{C(\l)} \to \delta_{C(\l)} [2]$ which we have when $h=1$. In the same way we can define $T(\l)$. The fact that they satisfy the nil affine Hecke relations is then immediate because they satisfy these relations when $h=1$.

\section{Associated graded of the $\sl_2$ action}\label{sec:gradedaction}

Consider the 2-category $ \K_\O $ with the same objects as $ \K_\D $ but morphism categories
$$ \K_\O(\l, \l') = D(\O_{T^* \G(k,N) \times T^* \G(k',N)}\mod^{\C^\times})$$
where $\l=N-2k,\l'=N-2k'$ as before. We have a functor
$$ \tgr : D(\D_{\G(k,N) \times \G(k',N),h}\mod) \rightarrow D(\O_{T^* \G(k,N) \times T^* \G(k',N)}\mod^{\C^\times})$$
and, by Proposition \ref{prop:1}, these functors combine to give an associated graded 2-functor $ \gr : \K_\D \rightarrow \K_\O$. Thus, the categorical $ \sl_2$ action defined in the previous section with target $ \K_\D $ gives rise to a categorical $ \sl_2$ action with target $ \K_\O $.  In this section, we will compute this action explicitly.

We will compute $\tgr(\delta_{C^r(\l),h})$ where $C^r(\l)$ is the correspondence from section \ref{sec:divided} (technically we only need the case $r=1$ but it is interesting to compute this for all $r$). Recall that
$$T^* \G(k,N) = \{(X,V): X \in \End(\C^N), X(V) = 0, X(\C^N) \subset V\}$$
and likewise for $T^* \G(k-r,N)$. Moreover, the conormal bundle of $C^r(\l) \subset \G(k,N) \times \G(k-r,N)$ can be identified with
$$\fC^{r}(\l) := \{(X,V,V'): 0 \subset V' \subset V \subset \C^N, X(\C^N) \subset V, X(V) \subset V', X(V')=0 \}.$$

\begin{Proposition}
As a kernel inside $T^* \G(k,N) \times T^* \G(k-r,N)$ we have
$$\tgr(\delta_{C^r(\l),h}) \cong \O_{\fC^r(\l)} \otimes \det(V)^{-k+r} \otimes \det(V')^{k} [-d_k] \{r(k-r) + d_k \}$$
while, as a kernel inside $T^* \G(k-r,N) \times T^* \G(k,N)$, we have
$$\tgr(\delta_{C^r(\l),h}) \cong \O_{\fC^r(\l)} \otimes \det(V'/V)^{N-k} \otimes \det(\C^N/V')^{-r} [-d_{k-r}] \{r(N-k)+ d_{k-r}\}$$
using the usual convention that in the second case $V'$ is the vector bundle on $T^* \G(k,N)$.
\end{Proposition}\label{prop:assgradedsl2}

Here and below, $ V, V', \C^N $ denote vector bundles whose fibres are these vector spaces.  Note that $ \C^N $ is trivial as a vector bundle, but carries a non-trivial $ \C^\times $-action.

\begin{proof}
First we compute $\gr(\delta_{C^r(\l),h})$. Using Corollary \ref{cor:grofDelta} we find that
$$\gr(\delta_{C^r(\l),h}) \cong \O_{\fC^r(k,N)} \otimes \omega_{X/Y} \{\dim(C^r(\l)\}$$
where $X = C^r(\l)$, $Y = \G(k,N) \times \G(k-r,N)$ and $\dim(C^r(\l)) = r(k-r) + d_k$. Moreover, it is a standard exercise to calculate that
\begin{align*}
\omega_{\G(k,N)} &\cong \det(V)^N \otimes \det(\C^N)^{-k} \\
\omega_{\fC^r(k,N)} &\cong \det(V')^k \otimes \det(V)^{N-k+r} \otimes \det(\C^N)^{-k}.
\end{align*}
It then follows that
\begin{align*}
\gr(\delta_{C^r(\l),h})
&\cong \O_{\fC^r(\l)} \otimes \omega_{\fC^r(\l)} \otimes \omega_{\G(k,N) \times \G(k-r,N)}^\vee \{r(k-r) + d_k\} \\
&\cong \O_{\fC^r(\l)} \otimes \det(V)^{-k+r} \det(V')^{k-N} \det(\C^N)^{k-r} \{r(k-r) + d_k\}.
\end{align*}
Thus, as a kernel inside $T^* \G(k,N) \times T^* \G(k-r,N)$ we have
\begin{align*}
\tgr(\delta_{C^r(\l),h})
&\cong \gr(\delta_{C^r(\l),h}) \otimes \omega_{\G(k-r,N)} [-d_k] \\
&\cong \O_{\fC^r(\l)} \otimes \det(V)^{-k+r} \det(V')^k [-d_k] \{r(k-r) + d_k\}
\end{align*}
while as a kernel inside $T^* \G(k-r,N) \times T^* \G(k,N)$ we have
\begin{align*}
\tgr(\delta_{C^r(\l),h})
& \cong \gr(\delta_{C^r(\l),h}) \otimes \omega_{\G(k,N)} [-d_{k-r}] \\
& \cong \O_{\fC^r(\l)} \otimes \det(V)^{k-N} \det(V)^{N-k+r} \det(\C^N)^r [-d_{k-r}] \{r(N-k)+d_{k-r}\}.
\end{align*}
This agrees with what we wanted to show.
\end{proof}

\begin{Corollary}\label{cor:sl2action}
The associated graded of the categorical $\sl_2$ action with target $\K_\D$ is the action with target $\K_\O$ and kernels
\begin{align*}
& \tgr(\sE^{(r)}(\l)) = \O_{\fC^r(\l)} \otimes \det(V)^{-k+r} \otimes \det(V')^{k} \{r(k-r)\} \in D(\O_{T^* \G(k,N) \times T^* \G(k-r,N)}\mod^{\C^\times}) \\
& \tgr(\sF^{(r)}(\l)) = \O_{\fC^r(\l)} \otimes \det(V'/V)^{N-k} \otimes \det(\C^N/V')^{-r} \{r(N-k)\} \in D(\O_{T^* \G(k-r,N) \times T^* \G(k,N)}\mod^{\C^\times})
\end{align*}
where $\l = N-2k+r$.
\end{Corollary}

\begin{Remark}
In \cite{CKL1} we defined another categorical $\sl_2$ action with target $ \K_\O $.  The kernels which induce this action are not exactly the same as those from Corollary \ref{cor:sl2action} but the difference is only some line bundles which can essentially be accounted for by conjugation. On the other hand, the categorical $\sl_2$ action defined in \cite{CK4} does agree with the one above, after restricting to the open subsets $T^* \G(k,N) \subset Y(k,N-k)$ and matching up the notation ({\it cf.} Appendix A.2 of \cite{CK4}).
\end{Remark}

\section{The equivalence $\T$}\label{sec:T}

In this section we study the equivalence
$$\T: D(\D_{\G(k,N),h}\mod) \rightarrow D(\D_{\G(N-k,N),h}\mod)$$
induced by the categorical $\sl_2$ action from section \ref{sec:actionDmod}. For notational simplicity we assume $2k \le N$. For $s=0,\dots,k$ define
$$\Theta^s := \sF^{(N-k-s)} * \sE^{(k-s)} \in D(\D_{\G(k,N) \times \G(N-k,N),h}\mod).$$
From these we can define Rickard's complex $\Theta^{k}\la -k \ra \rightarrow \dots \rightarrow \Theta^1 \la -1 \ra \rightarrow \Theta^0$ where the differential $\Theta^s \la -s \ra \rightarrow \Theta^{s-1} \la -s+1 \ra$ is given by the composition
$$\sF^{(N-k-s)} * \sE^{(k-s)} \la -s \ra \rightarrow  \sF^{(N-k-s)} * \sF * \sE * \sE^{(k-s)} \la N-3s+1 \ra \rightarrow \sF^{(N-k-s+1)} * \sE^{(k-s+1)} \la -s+1 \ra$$
where the first arrow uses the adjunction and the second map consists of two projections.

\begin{Theorem}
The complex $\Theta^*$ has a unique right convolution $\sT := \Cone(\Theta^*)$ which induces an equivalence $\T: D(\D_{\G(k,N),h}\mod) \xrightarrow{\sim} D(\D_{\G(N-k,N),h}\mod)$.
\end{Theorem}
\begin{proof}
Uniqueness and existence of the convolution in this context follows from Lemma \ref{lem:convolution}. The fact it gives an equivalence is Theorem 2.8 of \cite{CKL1} (this result is analogous to the main result of \cite{CR}).
\end{proof}
\begin{Remark} Convolutions of complexes in a triangulated category were introduced by Orlov \cite{O}. They are defined by an iterated cone construction. We choose to work with right convolutions, meaning that we start the iterated cones on the right.
\end{Remark}
\begin{Remark}
The complex $\Theta^*$ is actually the dual of the complex we usually considered in (for instance) \cite{CKL1}. The complex in those cases is $\Theta^0 \la -k \ra \rightarrow \dots \rightarrow \Theta^{k-1} \la -1 \ra \rightarrow \Theta^k$.
\end{Remark}

For $s = 0, \dots, k$ recall the locus $Z_s \subset \G(k, N) \times \G(N-k, N)$ consisting of pairs $(V, V')$ such that $\dim (V \cap V') = s$ (notice that $Z_0$ is an open). Note that $ \dim (Z_s) = 2d_k - s^2 $. The following is the main result in this section.

\begin{Theorem} \label{thm:equivopen}
We have $\sT \cong \fG(j_* \delta_{Z_0,m}) \la d_k \ra$ where $j: Z_0 \rightarrow \G(k,N) \times \G(N-k, N) $ is the open embedding. Moreover, the weight filtration on $\fG(j_* \delta_{Z_0,m}) \la d_k \ra$ agrees with the filtration on $\sT$ coming from Rickard's complex.
\end{Theorem}

\begin{Remark}
Here Saito's pushforward of Hodge modules is very important in the statement of the theorem.  In particular, $\fG(j_* \delta_{Z_0,m}) $ and $ \int_j \delta_{Z_0,h} $ are \textit{not} isomorphic as $ \D_h $-modules.
\end{Remark}

By Proposition \ref{prop:pushHodgeD}, we immediately deduce the following corollary which is used in the work of Bezrukavnikov-Losev \cite{BL}.
\begin{Corollary}
We have $\sT \otimes_{\C[h]} \C_1 \cong \int_j \O_{Z_0} [d_k]$.
\end{Corollary}

\subsection{Filtrations and iterated cones}

The main step in proving Theorem \ref{thm:equivopen} is showing that the associated graded pieces in the weight filtration of $\fG(j_* \delta_{Z_0,m})$ are the same as the terms in Rickard's complex used to define $\sT$ as an iterated cone. To do this we first recall a few generalities regarding filtrations and iterated cones.

Consider an object $\sA$ equipped with an increasing filtration
$$0 \subset W_0 \sA \subset W_1 \sA \subset \dots \subset W_k \sA = \sA$$
where we denote the subquotients by $\gr_i \sA := W_i \sA / W_{i-1} \sA$. If $k=1$ then we have the exact triangle
$$\gr_0 \sA \rightarrow \sA \rightarrow \gr_1 \sA$$
where $\gr_0 \sA = W_0 \sA$. We can rewrite this triangle as $\sA \cong \Cone(\gr_1 \sA [-1] \rightarrow \gr_0 \sA)$. If $k > 1$ then we can repeat this argument to find that
\begin{equation}\label{eq:cone}
\sA \cong \Cone \left( \gr_k \sA [-k] \rightarrow \gr_{k-1} \sA [-k+1] \rightarrow \dots \rightarrow \gr_1 \sA [-1] \rightarrow \gr_0 \sA \right).
\end{equation}
where the expression on the right is an iterated cone. Iterated cones, also called convolutions, were introduced by Orlov \cite{O}. The definition is quite general and applies to any complex
$$A_k \xrightarrow{f_k} A_{k-1} \rightarrow \dots \rightarrow A_1 \xrightarrow{f_1} A_0$$
of objects in a triangulated category. The convolution of a complex $(A_\bullet,f_\bullet)$ may not exist and may not be unique. However, there are simple homological conditions under which both existence are assured.

\begin{Lemma}\cite[Prop.8.3]{CK1}\label{lem:convolution}
Let $(A_\bullet, f_\bullet)$ be a complex in an abstract triangulated category.
\begin{enumerate}
\item If $\Hom(A_{i+j+1} [j], A_i) = 0 $ for all $ i \ge 0, j \ge 1 $, then any two convolutions of $ (A_\bullet, f_\bullet) $ are isomorphic.
\item If $\Hom(A_{i+j+2} [j], A_i) = 0 $ for all $ i \ge 0, j \ge 1 $, then $(A_\bullet, f_\bullet)$ has a convolution.
\end{enumerate}
\end{Lemma}

\subsection{The associated graded of the weight filtration} \label{se:weightgraded}

Consider the weight filtration $W$ on $j_* \delta_{Z_0,m}$ and the associated graded
$$\gr^W_s(j_* \delta_{Z_0,m}) := W_s(j_* \delta_{Z_0,m})/W_{s-1}(j_* \delta_{Z_0,m}).$$

\begin{Proposition}\label{prop:grj}
We have $\fG(\bigoplus_s \gr^W_s(j_* \delta_{Z_0,m})) \otimes_{\C[h]} \C_1 \cong \bigoplus_s \IC_{\oZ_s}$.
\end{Proposition}
\begin{proof}
Each $\gr^W_s(j_* \delta_{Z_0,m})$ is a polarizable pure Hodge module of weight $s$. Therefore,
$$\fG(\gr^W_s(j_* \delta_{Z_0,m})) \otimes_{\C[h]} \C_1$$
is isomorphic to a direct sum of IC $\D$-modules (cf. \cite{PS},  Theorem 14.37). Because everything is $GL_N$-equivariant, these must be the IC objects of $GL_N $ orbits. As all these orbits are simply connected, the only local systems arising must be trivial. Thus we obtain
\begin{equation} \label{eq:Phi}
\fG(\bigoplus_s \gr^W_s(j_* \delta_{Z_0,m})) \otimes_{\C[h]} \C_1 \cong \bigoplus_s \IC_{\oZ_s}^{\oplus f_s}
\end{equation}
for some $f_s \in \N$. It remains to show that all $f_s=1$.

Fix $ 0 \le \ell \le k$ and pick a point $ x_\ell \in Z_\ell$ and let $ i_\ell :\{ x_\ell \} \rightarrow \G(k,N) \times \G(N-k,N) $ denote the inclusion of this point. Using base change for $\D$-modules and Proposition \ref{prop:pushHodgeD} we have
\begin{equation}\label{eq:A}
i_\ell^{\d}(\fG(j_* \delta_{Z_0,m}) \otimes_{\C[h]} \C_1) \cong i_\ell^{\d} \int_j \delta_{Z_0} \cong \begin{cases} \C[-2d_k] & \text{ if $ \ell = 0 $} \\ 0 & \text{ otherwise } \end{cases}
\end{equation}

On the other hand, applying $i_\ell^{\d}$ to both sides of \eqref{eq:Phi}, we obtain
\begin{align}\label{eq:B}
i_\ell^{\d} \left( \bigoplus_s \fG(\gr^W_s(j_* \delta_{Z_0,m})) \otimes_{\C[h]} \C_1 \right)
&\cong \bigoplus_s i_\ell^{\d} \IC_{\oZ_s}^{\oplus f_s} \\
\nonumber &\cong \bigoplus_s \bigoplus_{\qbins{\ell}{s}} \C^{\oplus f_s} [s\ell-2d_k]
\end{align}
where the second isomorphism follows from Lemma \ref{lem:fibre} (restricted to $h=1$). Taking Euler characteristics gives
$$\sum_s (-1)^{s\ell} f_s (-1)^{s(\ell-s)} \binom{\ell}{s}  = \begin{cases} 1 & \text{ if $ \ell = 0 $} \\ 0 & \text{ otherwise } \end{cases}$$
The upper-triangularity of this system of relations above means that there is at most one set of solutions $f_s$. Since
$$\sum_{s} (-1)^s \binom{\ell}{s} = \begin{cases} 1 & \text{ if $ \ell = 0 $} \\ 0 & \text{ otherwise} \end{cases}$$
it follows that $f_s = 1$.
\end{proof}

\begin{Corollary}\label{cor:grj}
For each $ s = 0,\dots, k $, we have $\gr^W_s(j_* \delta_{Z_0,m}) \cong \IC_{\oZ_s,m} \{s\}$ and for $ p> s$, $\gr^W_p( j_* \delta_{Z_0,m}) = 0 $.
\end{Corollary}
\begin{proof}
By Saito's results each $\gr^W_p(j_* \delta_{Z_0,m})$ is a pure polarizable Hodge module of weight $ p $.  By the structure theorem, we have
$$ \oplus_p \gr^W_s(j_* \delta_{Z_0,m}) = \oplus_t M_{\oZ_t} $$
where $ M_{\oZ_t} $ has strict support on $\oZ_t$.  On the other hand, from Proposition \ref{prop:grj}, we know that the underlying $\D$-module of $ \oplus_s \gr^W_s(j_* \delta_{Z_0,m}) $ is a direct sum of IC objects one for each stratum.  Thus if we restrict to an open subset of $ \oZ_0$, we see that $ M_{\oZ_0} $ must restrict to a constant rank 1 variation of Hodge structure (since the direct sum of the IC $ \D$-modules restricts to a trivial rank 1 flat connection).  Thus $M_{\oZ_0} = \IC_{\oZ, m}\{e_0\} $ for some integer $ e_0 $.  We then remove this direct summand and continue.  In this way, we conclude that $ M_{\oZ_s} = \IC_{\oZ_s,m}\{e_s\}$ for some integers $ e_s$.

To determine the integers $ e_s $, we just repeat the computation from the proof of Proposition \ref{prop:grj} except working in the category of mixed Hodge modules.  This gives us the equation

$$\sum_s (-1)^{e_s}  q^{s\ell -2d_k- e_s} \qbins{\ell}{s}  = \begin{cases} q^{-2d_k} & \text{ if $ \ell = 0 $} \\ 0 & \text{ otherwise } \end{cases}$$

The identity
\begin{equation*}
\sum_{s} (-1)^s q^{s \ell-s} \qbins{\ell}{s} = \begin{cases} 1 & \text{ if $ \ell = 0 $} \\ 0 & \text{ otherwise} \end{cases}
\end{equation*}
thus implies that  $e_s = s$.

Hence we see that
$$\oplus_p \gr^W_p(j_* \delta_{Z_0,m}) = \oplus_{s=0}^k \IC_{\oZ_s, m} \{s \} $$
Since for each $ p $,  there is at most one summand of weight $ p $ on the right hand side, the result follows.

\end{proof}

\begin{Lemma}\label{lem:fibre}
We have $i_\ell^! \IC_{\oZ_s,m} \cong \bigoplus_{\qbins{\ell}{s}} \C \la s\ell-2d_k \ra$.
\end{Lemma}
\begin{proof}
Recall the variety
$$P_s = \{ (V'', V, V') \in \G(s, N) \times \G(k,N) \times \G(N-k, N): V'' \subset (V \cap V')\}.$$
The natural map $P_s \rightarrow \G(k,N) \times \G(N-k, N) $ has image $\oZ_s$ and by Proposition \ref{pr:smallres} the map $\pi_s: P_s \rightarrow \oZ_s$ is a small resolution. This means that $\IC_{\oZ_s,m} = \pi_{s*} \delta_{P_s,m}$. Applying base change we get $i_\ell^{!} \IC_{\oZ_s,m} = \pi_{s*} \delta_{\pi_s^{-1}(x_\ell),m} \la \dim \pi_s^{-1}(x_\ell) - \dim P_s \ra$. Note that
$$ \dim (\pi_s^{-1}(x_\ell)) - \dim (P_s) = s(\ell-s) - (2d_k - s^2) = s\ell - 2d_k.$$
Now $\pi_s^{-1}(x_\ell) \cong \G(s,\ell)$ and thus
$$i_\ell^! \IC_{\oZ_s,m} \cong H^* (\G(s,\ell)) \la s\ell-2d_k \ra \cong \bigoplus_{\qbins{\ell}{s}} \C \la s\ell -2d_k \ra.$$
\end{proof}

\subsection{Proof of Theorem \ref{thm:equivopen}}

\begin{Proposition} \label{prop:ThetaIC}
We have $\Theta^s \cong \IC_{\oZ_s,h} \la d_k \ra$.
\end{Proposition}
\begin{proof}
Applying similar reasoning as in the proof of Proposition \ref{prop:sl2Grass} we see that
$$\sF^{(N-k-s)} * \sE^{(k-s)} = \int_\pi \delta_{P_s,h} \la d_k \ra.$$
By Proposition \ref{pr:smallres}, we know that $ \int_\pi \delta_{P_s,h} = \IC_{\oZ_s,h}$. The result follows.
\end{proof}

We will now complete the proof of Theorem \ref{thm:equivopen}.   Recall that $\sT$ is the right convolution of $\Theta^k \la -k \ra \rightarrow \dots \rightarrow \Theta^0$.  By Proposition \ref{prop:ThetaIC}, we can rewrite this complex (up to an overall $\la d_k \ra$) as
$$
\IC_{\oZ_k,h} \la -k \ra \rightarrow \dots \rightarrow \IC_{\oZ_1,h} \la -1 \ra \rightarrow \IC_{\oZ_0,h}
$$
By Lemma \ref{lem:j*} below, we know that the differentials in this complex are unique (up to scalar).

 Denote $U_s := Z_0 \cup Z_1 \dots \cup Z_s$ and let $j_s: Z_0 \rightarrow U_s$ be the natural inclusion (note that $j_k = j$).  We will show by induction that $\fG(j_{s*} \delta_{Z_0,m})$ is the iterated cone of a complex
\begin{equation}\label{eq:complex}
\IC_{\oZ_s,h} \la -s \ra \rightarrow \dots \rightarrow \IC_{\oZ_1,h} \la -1 \ra \rightarrow \IC_{\oZ_0,h}
\end{equation}
restricted to $U_s$. The base case is trivial. By Corollary \ref{cor:grj} we know that
$$\bigoplus_i \fG(\gr_i^W (j_{s*} \delta_{Z_0,m})) \cong \bigoplus_{i=1}^s \IC_{\oZ_s,h} \{s\}.$$
Moreover, the weight filtration of $j_{s*} \delta_{Z_0,m}$ restricted to $U_{s-1}$ agrees with the weight filtration of $j_{(s-1)*} \delta_{Z_0,m}$. These two facts together with the induction hypothesis imply that $\fG(j_{s*} \O_{Z_0,m})$ is isomorphic to the iterated cone of a complex of the form
$$\IC_{\oZ_{s-1},h} \la -s+1 \ra  \rightarrow \dots \rightarrow \IC_{\oZ_i,h} \la -i \ra  \oplus \IC_{\oZ_s,h} [-i] \{s\} \rightarrow \dots \rightarrow \IC_{\oZ_0,h}$$
restricted to $U_s$, for some $i$.
However, if $i \ne s$ then by Lemma \ref{lem:j*} all the maps to or from $\IC_{\oZ_s,h}[-i]\{s\}$ would be $\C[h]$-torsion and hence $\fG(j_{s*} \delta_{Z_0,m}) \otimes_{\C[h]} \C_1 \cong j_{s*} \O_{Z_0}$ would be decomposable (contradiction). It follows that $i=s$ and hence $\fG(j_{s*} \delta_{Z_0,m})$ is of the form (\ref{eq:complex}).

\begin{Lemma}\label{lem:j*}
We have that
\begin{itemize}
\item $\Hom_{\C[h]}(\IC_{\oZ_j,h} \{j\} [-l], \IC_{\oZ_i,h} \{i\})$ is $\C[h]$-torsion if $l \ne j-i$ and
\item $\Hom_{\C[h]}(\IC_{\oZ_{i+1},h} \la -1 \ra, \IC_{\oZ_i,h}) \cong \C[h]$.
\end{itemize}
\end{Lemma}
\begin{proof}
Using Proposition \ref{prop:ThetaIC} we can identify $\IC_{\oZ_i,h}$ with $\sF^{(N-k-i)} * \sE^{(k-i)}$ (up to a shift by $\la d_k \ra$ which does not depend on $i$ and we will ignore). On the other hand, by applying adjunction and the $\sl_2$ commutator relation repeatedly one can decompose
\begin{equation}\label{eq:6}
\Hom_{\C[h]}(\sF^{(N-k-j)} * \sE^{(k-j)} \{j\} [-l], \sF^{(N-k-i)} * \sE^{(k-i)} \{i\})
\end{equation}
as a direct sum of terms of the form $\Hom_{\C[h]}(\id_{\mu},\id_{\mu} [l]\{i-j\} \la a \ra)$ for various $\mu$ and $a \in \Z$. On the other hand
$$\Hom_{\C[h]}^*(\id_\mu,\id_\mu) = \Hom_{\C[h]}^*(\int_{\Delta} \delta_{\G(k,N),h}, \int_{\Delta} \delta_{\G(k,N),h})$$
which, using adjunction, is (modulo $\C[h]$-torsion) isomorphic to a direct sum of terms of the form $\C[h] \la a \ra$. Since $\la a \ra = [a]\{-a\}$, (\ref{eq:6}) is $\C[h]$-torsion unless $l+i-j=0$. This proves the first assertion.

The second assertion is proved similarly by using the $\sl_2$ commutator relation to simplify the Hom space (see for instance \cite{CKL1}). In the end we end up with a direct sum of terms of the form $\Hom_{\C[h]}^*(\id_\mu,\id_\mu)$ where $* \le 0$ (with only one term where $*=0$). The result then follows since
$$\Hom_{\C[h]}^*(\id_\mu,\id_\mu) \cong \begin{cases} \C[h] & \text{ if } * = 0 \\ 0 & \text{ if } * < 0 \end{cases}$$
\end{proof}

\begin{Remark}
It is possible to show directly that $j_* \delta_{Z_0,m}$ is a kernel which induces an equivalence without relating it to a categorical $\sl_2$ action. This fact is well-known to experts though we were not able to find a proof in the literature. On the other hand, this still does not immediately imply that $\fG(j_* \delta_{Z_0,m})$ induces an equivalence between categories of $\D_h$-modules.
\end{Remark}

\begin{Remark}
Instead of $j_* \delta_{Z_0,m}$ we can equally well have considered $j_! \delta_{Z_0,m}$. In this case one can show that $\sT^{-1} \cong \fG(j_! \delta_{Z_0,m}) \la d_k \ra$.
\end{Remark}

\section{Associated graded of $\T$}\label{sec:grT}

In Theorem \ref{thm:equivopen} we saw that the equivalence $\T$ is induced by $\sT \cong \fG(j_* \O_{Z_0,m}) \la d_k \ra$. The associated graded $\tgr(\fG(j_* \O_{Z_0,m}))$ can be described as the convolution of the complex $\tgr(\Theta^*)$ or equivalently, coming from the categorical $\sl_2$ action defined in Corollary \ref{cor:sl2action}.

More precisely, recall from \cite{C1} the locus inside $T^* \G(k,N) \times T^* \G(N-k,N)$ given by
$$\fZ(k,N) := \{(X,V,V'): \dim(V)=k, \dim(V')=N-k, X(\C^N) \subset V \cap V', X(V)=X(V')=0 \}$$
where $V,V' \subset \C^N$ and $X \in \End(\C^N)$. This is just the fiber product of $T^* \G(k,N)$ and $T^* \G(N-k,N)$ over their common affinization. This variety consists of $k+1$ irreducible components $\fZ_s(k,N)$ where $s=0, \dots, k$ and $\fZ_s(k,N)$ is the locus where
$$\dim(\ker X) \ge N-s \ \ \text{ and } \ \ \dim(V \cap V') \ge s.$$
Note that $\fZ_s$ is the closure of the conormal bundle to $Z_s$.

\begin{Proposition}\label{prop:grtheta}
For each $\Theta^s$ we have
$$\tgr(\Theta^s) \cong \widetilde{\O}_{\fZ_s(k,N)} \otimes \det(V)^{-s} \otimes \det(V')^{N-s} \otimes \det(\C^N)^{-N+k+s} \{d_k-s^2\}$$
where $\widetilde{\O}_{\fZ_s(k,N)}$ denotes the normalized structure sheaf.
\end{Proposition}
\begin{Remark} We expect that $\fZ_s(k,N)$ is normal in which case $\widetilde{\O}_{\fZ_s(k,N)} = \O_{\fZ_s(k,N)}$.
\end{Remark}
\begin{proof}
By Proposition \ref{prop:1} we have
$$\tgr(\sF^{(N-k-s)} * \sE^{(k-s)}) \cong \tgr(\sF^{(N-k-s)}) * \tgr(\sE^{(k-s)}).$$
Both these terms are identified in Corollary \ref{cor:sl2action} and they agree with the categorical $\sl_2$ action from \cite{CK4}. The result now follows from Proposition A.7 of \cite{CK4} keeping in mind that $\fZ_s(k,N)$ corresponds to $Z_{k-s}(k,N-k)$ in the notation from \cite{CK4}.
\end{proof}

Thus, $\tgr(\sT)$ is the convolution of the complex made up of $\tgr(\Theta^s)$ which are described in Proposition \ref{prop:grtheta}. Although this description of $\tgr(\sT)$ is quite useful to work with we now give another description which more closely resembles the definition of $\sT$ as the pushforward $j_* \O_{Z_0,m}$. It would be interesting to relate directly these two descriptions of $\sT$ and $\tgr(\sT)$.

We begin by defining open subschemes $\fZ_s^o(k,N) \subset \fZ_s(k,N)$ by the condition
$$\dim(\ker X) + \dim(V \cap V') \le N+1.$$
We denote by $\fZ^o(k,N) \subset \fZ(k,N)$ the union of all $\fZ_s^o(k,N)$ and by $f: \fZ^o \hookrightarrow \fZ$ their inclusion. The advantage of $\fZ^o$ is that it avoids the more complicated singularities of $\fZ$ but is big enough that the complement $\fZ \setminus \fZ^o$ has codimension at least two.

\begin{Lemma}\cite[Lemma 3.4]{C1}
The intersection $D^o_{s,+}(k,N) := \fZ_s^o(k,N) \cap \fZ_{s+1}(k,N)$ is a Cartier divisor in $\fZ_s^o(k,N)$ and corresponds to the locus where $\dim \ker(X) = N-s$ and $\dim(V \cap V') = s+1$.
\end{Lemma}

\begin{Proposition}
On $\fZ^o(k,N)$ there exists a line bundle $\sL(k,N)$ uniquely determined by its restriction
$$\sL(k,N)|_{\fZ_s^o(k,N)} \cong \O_{\fZ_s^o(k,N)}([D^o_{s,+}(k,N)]) \otimes \det(V)^{-s} \otimes \det(V')^{N-s} \otimes \det(\C^N)^{-N+k+s} \{-s(s-1)\}$$
to each component of $\fZ^o(k,N)$. Moreover, $\tgr(\sT) \cong R^0 f_* (\sL(k,N)) \{d_k\}$.
\end{Proposition}
\begin{proof}
This result follows from Proposition A.8 of \cite{CK4} with a minor amount of work. \end{proof}

\begin{Corollary}\label{cor:main}
We have $\gr(\fG(j_* \O_{Z_0,m})) \cong R^0 f_*(\sL'(k,N))$ where $\sL'(k,N)$ is determined by its restrictions
$$\sL'(k,N)|_{\fZ_s^o(k,N)} \cong \O_{\fZ_s^o(k,N)}([D^o_{s,+}(k,N)]) \otimes \det(V)^{-s} \otimes \det(\C^N/V')^{s} \{-s(s-1)\}.$$
\end{Corollary}
\begin{proof}
This follows from Theorem \ref{thm:equivopen} and the last assertion of the previous proposition once we unpack the definitions of $ \tgr $ and $ \sT $. In particular, $\sL'(k,N) = \sL(k,N) \otimes f^* \omega_{\G(N-k,N)}^\vee$ which gives the expression above for $\sL'(k,N)$.
\end{proof}

\section{Generalization to cominuscule flag varieties}\label{sec:cominus}

In this section, we outline a possible generalization of our results to cominuscule flag varieties.  Other than the existence of a categorical $ \sl_2$ action, we expect that the results of this paper generalize nicely to this setting.

\subsection{Geometric setup}
Let $G $ be a simple complex algebraic group.  Recall that a node $ i $ of the Dynkin diagram of $ G $ is called cominuscule if, for every positive root $\alpha $, the coefficient of $ \alpha_i $ in $ \alpha $ is at most $ 1$.  Let $ i $ be a cominuscule node and let $ P $ be the associated maximal parabolic subgroup.  In this case, $G/P$ is called a cominuscule flag variety.  Let $ Q = P^T $ be the opposite parabolic subgroup (i.e. $\g_\beta \subset \mathfrak q $ iff. $ \g_{-\beta} \subset \mathfrak p $).  $ G/Q $ is also a cominuscule flag variety.

Denote by $W$ the Weyl group of $G$ and let
$$ A = \{ w \in W : w < s_j w\text{ and } w < w s_j, \text{ for all } j \ne i \}.$$
Sometimes these are called biGrassmannian elements of $W$. $A$ carries a partial order by restricting the Bruhat order of $ W $.
A standard result tells us that there is an order-reversing bijection between $ A$ and the set of $ G $-orbits on $ G/P \times G/Q $.  Also for any element $ w$ of $ A $ the corresponding orbit has codimension $ \ell(w)$.

We have checked the following lemma in a case-by-case fashion.  We believe that there should be a general proof.

\begin{Lemma}
The set $ A$ is linearly ordered.
\end{Lemma}

Let $k +1 $ be the cardinality of $ A$ and let  $ Z_0, \dots, Z_k $ be the $ G $ orbits on $ G/P \times G/Q $, in decreasing size, so $ Z_0 $ denotes the open orbit.

\subsection{Some examples}
\subsubsection{Usual Grassmannians}
If we take $ G = SL_N $, then we can choose any node $i $ of the diagram and we get $ G/P = \G(i,N), G/Q = \G(N-i,N) $ and we return to our previous setup (in this case, $ k = \min(i, N - i) $).
\subsubsection{Even dimensional quadric}
If we take $ G = SO_{2n}$ and $ i = 1 $, then $ G/P = \{[x_1, \dots, x_{2n}] : x_1^2 + \cdots + x_{2n}^2 = 0 \} $ is a quadric in $ \P^{2n-1} $.  In this case $ G/Q = G/P $ and $ k = 2$.  Moreover, we have a simple description of the $G$-orbits:
$$ Z_0 = \{([x],[y]) : x \cdot y \ne 0 \}, \ Z_1 = \{([x],[y]) : x \cdot y = 0, x \ne y \}, \ Z_2 = \{([x],[y]) : x = y \}. $$
\subsubsection{Lagrangian Grassmannian}
Assume that $ n $ is odd.  If we take $ G = SO_{2n}$ and $ i = n $, then $ G/P = OG(n,2n)_+ $, one connected component of the variety of Lagrangian subspaces of $ \C^{2n}$ (with respect to the symmetric bilinear form).  In this case, $ G/Q = OG(n,2n)_- $, the other connected component.  In this case $ k = \frac{n-1}{2} $ and
\begin{equation*}
\begin{gathered}
Z_0 = \{(V,W) : \dim (V \cap W) = 0 \}, \ Z_1 = \{(V,W) : \dim (V \cap W) =  2 \},\\
\dots, \ Z_k = \{(V,W) : \dim (V \cap W) = n - 1 \}.
\end{gathered}
\end{equation*}
Note that for $ V \in OG(n,2n)_+, W \in OG(n,2n)_- $, $\dim (V \cap W)$ is always even.

\subsection{The $ \D$-module equivalence}
Let $ j : Z_0 \rightarrow G/P \times G/Q $ be the inclusion.  Note that $ \overline{Z_1} $ is a divisor in $ G/P \times G/Q$, since it is associated to the element $ s_i \in A $ which has length 1.  Thus $ j $ is an affine morphism.  Thus, $ \int_j \O_{Z_0} $ is a $ \D_{G/P \times G/Q} $-module (there is no higher direct image). The following result should not be difficult to establish.

\begin{Conjecture}
The kernel $ \int_j \delta_{Z_0} $ induces an equivalence $ D(\D_{G/P}\mod) \xrightarrow{\sim} D(\D_{G/Q}\mod) $.
\end{Conjecture}

We may also consider the Hodge module push-forward $ j_* \delta_{Z_0,m}  $ and its underlying $ \D_h$-module $ \sT = \fG(j_* \delta_{Z_0,m})\la \dim (G/P) \ra$. We can extend the previous conjecture as follows.

\begin{Conjecture}
The kernel $\sT $ induces an equivalence $ D(\D_{G/P,h}\mod) \xrightarrow{\sim} D(\D_{G/Q,h}\mod) $.
\end{Conjecture}

As in section \ref{se:weightgraded}, we can consider the weight filtration.  We hope that similar techniques as in that section will help to establish the following result.

\begin{Conjecture} \label{co:weightgraded}
For $ s = 0, \dots, k$, we have $gr^W_s(j_* \delta_{Z_0,m}) = \IC_{\oZ_s,m}\{s \} $ and the rest of the subquotients are zero.
\end{Conjecture}

\subsection{$\O$-module side}
Now, we consider kernels on the product of the cotangent bundles $ T^* G/P \times T^* G/Q $. For each $ s$, let $ \fZ_s = \overline{T^*_{Z_s}(G/P \times G/Q)} \subset T^* G/P \times T^* G/Q $ be the closure of the conormal bundle of $ Z_s $.  These components $ \fZ_0, \dots, \fZ_k $ are the irreducible components of $ \fZ = T^* G/P \times_{\mathfrak g} T^* G/Q $.

Let $ Q_s = \tgr(\IC_{\oZ_s,h}\la \dim (G/P) \ra\{s\}) $.  If we assume the conjectures of the previous section, then we obtain the following.

\begin{Corollary}
The kernel $\tgr(\sT) $ induces an equivalence $D(\O_{T^* G/P}\mod) \rightarrow D(\O_{T^* G/Q}\mod) $. Moreover, there is a filtration of $ \tgr(\sT)$ whose subquotients are the sheaves $ Q_s $.
\end{Corollary}

We can try to give a more intrinsic construction of $ \tgr(\sT) $ as follows.  First, we have the following conjectural description of the sheaves $ Q_s$.

\begin{Conjecture}
The sheaf $ Q_s $, supported on $\fZ_s $, is the pushforward of a line bundle on the normalization of $ \fZ_s $.
\end{Conjecture}

As in the $T^* \G(k,N) $ case, we hope that this will lead to a nice description of  $\tgr(\sT) $ .

\begin{Conjecture}
There exists an dense subset $ \fZ^o $ of $ \fZ $ and a line bundle $ \sL $ on $ \fZ^o$, such that $ \tgr(\sT) \cong R^0 f_* \sL $ (where $ f : \fZ^o \rightarrow \fZ $ is the inclusion).
\end{Conjecture}

Of course, it would be nice to explicitly describe this open subset and this line bundle.


\begin{thebibliography}{E-G-S}

\bibitem[BLM]{BLM}
A. Beilinson, G. Lusztig and R. MacPherson, A geometric setting for the quantum deformation of $GL_n$, \textit{Duke Math. J.} \textbf{61} (1990), no. 2, 655--677.

\bibitem[BL]{BL}
R. Bezrukavnikov and I. Losev, Etingof conjecture for quantized quiver varieties, {\sf arXiv:1309.1716}.

\bibitem[Bj]{Bj}
J. Bjork, Analytic $\mathcal{D}$-modules and applications, \textit{Mathematics and its Applications}, \textbf{247}, Kluwer Academic Publishers, 1993

\bibitem[Br]{Br}
J. Brundan, On the definition of Kac-Moody 2-category, \textit{Math. Ann.} \textbf{364} (2016), no. 1-2, 353--372; {\sf arXiv:1501.00350}.

\bibitem[C1]{C1}
S. Cautis, Equivalences and stratified flops, \textit{Compositio Math.} \textbf{148} (2012), no. 1, 185--209; {\sf arXiv:0909.0817}.

\bibitem[C2]{C2}
S. Cautis, Rigidity in higher representation theory; {\sf arXiv:1409.0827}.

\bibitem[CK1]{CK1}
S. Cautis and J. Kamnitzer, Knot homology via derived categories of coherent sheaves I, $\sl_2$ case, \textit{Duke Math. J.} \textbf{142} (2008), no. 3, 511--588; \textsf{math.AG/0701194}.

\bibitem[CK2]{CK2}
S. Cautis and J. Kamnitzer, Knot homology via derived categories of coherent sheaves II, \textit{Invent. Math.} \textbf{174} (2008), no. 1, 165--232; \textsf{arXiv:0710.3216}.

\bibitem[CK3]{CK3}
S. Cautis and J. Kamnitzer, Braiding via geometric categorical Lie algebra actions, \textit{Compositio Math.} \textbf{148} (2012), no. 2, 464--506; \textsf{arXiv:1001.0619}.

\bibitem[CK4]{CK4}
S. Cautis and J. Kamnitzer, Knot homology via derived categories of coherent sheaves IV, coloured links, \textit{Quantum Topology}, \textbf{8} (2017), no. 2, 381--411; \textsf{arXiv:1410.7156}.


\bibitem[CKL1]{CKL1}
S. Cautis, J. Kamnitzer and A. Licata, Derived equivalences for cotangent bundles of Grassmannians via categorical $\sl_2$ actions, \emph{J. Reine Angew. Math.} \textbf{675} (2013), 53--99; \textsf{arXiv:0902.1797}.

\bibitem[CKL2]{CKL2}
S. Cautis, J. Kamnitzer and A. Licata, Coherent sheaves on quiver varieties and categorification, \emph{Math. Ann.} \textbf{357} (2013), no. 3, 805--854; \textsf{arXiv:1104.0352}.

\bibitem[CL]{CL}
S. Cautis and A. Lauda, Implicit structure in 2-representations of quantum groups, \textit{Selecta Math.} \textbf{21} (2015), no. 1, 201--244; \textsf{arXiv:1111.1431}.

\bibitem[CR]{CR}
J. Chuang and R. Rouquier, Derived equivalences for symmetric groups and $\sl_2$-categorification; \textit{Annals of Mathematics}, \textbf{167} (2008), 245-298. \textsf{math.RT/0407205}.

\bibitem[Co]{Co}
B. Conrad, Grothendieck Duality and Base Change, \textit{Lecture Notes in Mathematics}, \textbf{1750}, Springer Verlag, 2000.



\bibitem[Ha]{HartRD}
R. Hartshorne, Residues and Duality, \textit{Lecture Notes in Mathematics}, \textbf{20}, Springer Verlag, 1966.

\bibitem[HTT]{HTT}
R. Hotta, K. Takeuchi, and T. Tanisaki, $\D$-modules, perverse sheaves, and representation theory, \textit{Progress in Mathematics}, \textbf{236}, Birkhauser Boston, 2008.

\bibitem[Ld]{Ld}
A. Lauda, A categorification of quantum $\sl_2$, \textit{Adv. in Math.}, Vol. 225, Issue 6 (2010), 3327--3424; \textsf{arXiv:0803.3652v2}.

\bibitem[Lm]{Lm}
G. Laumon, Sur la cat\'egorie d\'eriv\'ee des $\D$-modules filtr\'es, \textit{Lecture Notes in Math.} {\bf 1016}, Springer, Berlin 1983, 151--237.

\bibitem[O]{O}
D. Orlov, Equivalences of derived categories and K3 surfaces, \textit{J. Math. Sci. (New York)} \textbf{84} (1997), no. 5, 1361--1381.


\bibitem[PS]{PS}
Chris Peters and Joseph Steenbrink, Mixed Hodge structures, \textit{Ergebnisse der Mathematik und ihrer Grenzgebiete}. 3. Folge. \textit{A Series of Modern Surveys in Mathematics}, {\bf 52}, Springer-Verlag, Berlin, 2008.

\bibitem[Sa1]{Sa1}
M. Saito, Modules de Hodge polarisables, \textit{Publ. RIMS. Kyoto Univ.} {\bf 24} (1988), 849--995.

\bibitem[Sa2]{Sa2}
M. Saito, Mixed Hodge Modules, \textit{Publ. RIMS. Kyoto Univ.} {\bf 26} (1990), 221--333.

\bibitem[Sa3]{Sa3}
M. Saito, Induced $\mathcal{D}$-modules and Differential Complexes, \textit{Bulletin de la S.M.F.}, tome 117, no.3 (1989), 361-387.

\bibitem[Sc]{Sc}
C. Schnell, An overview of Morihiko Saito's theory of mixed Hodge modules; {\sf arXiv:1405.3096}.

\bibitem[Stacks]{Stacks}
The Stacks Project Authors, The Stacks Project, {\url{https://stacks.math.columbia.edu}}, 2018.

\bibitem[W]{W}
B. Webster, A categorical action on quantized quiver varieties; {\sf arXiv:1208.5957}.

\bibitem[WW]{WW}
B. Webster and G. Williamson, A geometric construction of colored HOMFLYPT homology; \textsf{arXiv:0905.0486}.


\bibitem[Ze]{Ze}
A. Zelevinsky, Small resolutions of singularities of Schubert varieties, \textit{Functional Anal. Appl.} \textbf{17} (1983), no. 2, 142–-144.

\bibitem[Z1]{Z1}
H. Zheng, A geometric categorification of tensor products of $U_q(\sl_2)$-modules; {\sf arXiv:0705.2630}.

\bibitem[Z2]{Z2}
H. Zheng, Categorification of integrable representations of quantum groups, \textit{Acta Math. Sin.}, \textbf{30} (2014), no. 6, 899--932; {\sf arXiv:0803.3668}.

\end{thebibliography}
\end{document}